\documentclass[11pt]{amsart}
\usepackage{amssymb,amsthm,amsmath,amstext}
\usepackage{mathrsfs}  
\usepackage{bm}        
\usepackage{mathtools} 
\usepackage{color}
\usepackage[bbgreekl]{mathbbol}
\usepackage{cite}

\usepackage[left=1.5in,top=1in,right=1.5in,bottom=1in]{geometry}

\usepackage{graphicx}

\usepackage[all]{xy}
\usepackage{tikz}
\usetikzlibrary{snakes}

\newcommand{\xycenter}[1]{
	\begin{center}
	\mbox{\xymatrix{#1}}
	\end{center}
	}

\theoremstyle{plain}
\newtheorem{theorem}{Theorem}[section]

\newtheorem{proposition}[theorem]{Proposition}

\newtheorem{corollary}[theorem]{Corollary}

\theoremstyle{definition}
\newtheorem{definition}[theorem]{Definition}

\newtheorem{example}[theorem]{Example}

\theoremstyle{remark}
\newtheorem{remark}[theorem]{Remark}

\newcommand{\ioUpper}[2]{\iota^*_{\{#1\}>\{#2\}}}
\newcommand{\ioLower}[2]{\iota_{\{#1\}>\{#2\}*}}

\newcommand{\sheaf}[1]{\mathscr{#1}}
\newcommand{\DD}{\sheaf{D}}
\newcommand{\LL}{\sheaf{L}}
\newcommand{\OO}{\sheaf{O}}

\newcommand{\NN}{\sheaf{N}}
\newcommand{\PP}{\sheaf{P}}
\newcommand{\VV}{\sheaf{V}}
\newcommand{\ZZ}{\sheaf{Z}}

\newcommand{\YY}{\sheaf{Y}}
\newcommand{\XX}{\sheaf{X}}

\newcommand{\Z}{\mathbb Z}
\newcommand{\N}{\mathbb N}
\newcommand{\A}{\mathbb A}
\newcommand{\R}{\mathbb R}
\newcommand{\C}{\mathbb C}
\renewcommand{\P}{\mathbb P}
\newcommand{\Q}{\mathbb Q}

\DeclareMathOperator{\coker}{\mathrm{coker}}
\DeclareMathOperator{\Chow}{\mathrm{CH}}
\DeclareMathOperator{\Chownum}{\mathrm{CH}_{\mathrm{num}}}
\DeclareMathOperator{\Chowprelog}{\mathrm{CH}^*_{\mathrm{prelog}}}
\DeclareMathOperator{\Chowprelognum}{\mathrm{Num}^*_{\mathrm{prelog}}}
\DeclareMathOperator{\Chowprelogsat}{\mathrm{Chow}_{\mathrm{prelog, sat}}}
\DeclareMathOperator{\Chowprelognotdominant}{\mathrm{Chow}_{\mathrm{prelog, nd}}}
\DeclareMathOperator{\Numsatprelog}{\mathrm{Num}_{\mathrm{prelog, sat}}}

\newcommand{\im}{\mathrm{im}}

\usepackage[backref=page]{hyperref}
\hypersetup{pdftitle={Universal triviality of the Chow group of 0-cycles and the Brauer group}}
\hypersetup{pdfauthor={Christian Boehning, Hans-Christian Graf von Bothmer, Michel van Garrel}}
\hypersetup{colorlinks=true,linkcolor=blue,anchorcolor=blue,citecolor=blue,letterpaper=true}

\begin{document}


\title[Prelog Chow rings]{Prelog Chow rings and degenerations}



\author[B\"ohning]{Christian B\"ohning}
\address{Christian B\"ohning, Mathematics Institute, University of Warwick\\
Coventry CV4 7AL, England}
\email{C.Boehning@warwick.ac.uk}

\author[von Bothmer]{Hans-Christian Graf von Bothmer}
\address{Hans-Christian Graf von Bothmer, Fachbereich Mathematik der Universit\"at Hamburg\\
Bundesstra\ss e 55\\
20146 Hamburg, Germany}
\email{hcvbothmer@gmail.com}

\author[van Garrel]{Michel van Garrel}
\address{Michel van Garrel, School of Mathematics, University of Birmingham\\
Birmingham B15 2TT, England}
\email{m.vangarrel@bham.ac.uk}

\date{\today}


\begin{abstract}
For a simple normal crossing variety $X$, we introduce the concepts of prelog Chow ring, saturated prelog Chow group, as well as their counterparts for numerical equivalence. Thinking of $X$ as the central fibre in a (strictly) semistable degeneration, these objects can intuitively be thought of as consisting of cycle classes on $X$ for which some initial obstruction to arise as specializations of cycle classes on the generic fibre is absent. Cycle classes in the generic fibre specialize to their prelog counterparts in the central fibre, thus extending to Chow rings the method of studying smooth varieties via strictly semistable degenerations. After proving basic properties for prelog Chow rings and groups, we explain how they can be used in an envisaged further development of the degeneration method by Voisin et al. to prove stable irrationality of very general fibres of certain families of varieties; this extension would allow for much more singular degenerations, such as toric degenerations as occur in the Gross-Siebert programme, to be usable. 
We illustrate that by looking at the example of degenerations of elliptic curves, which, although simple, shows that our notion of prelog decomposition of the diagonal can also be used as an obstruction in cases where all components in a degeneration and their mutual intersections are rational.
 We also compute the saturated prelog Chow group of degenerations of cubic surfaces.
\end{abstract}

\maketitle

\tableofcontents

\section{Introduction}\label{sIntroduction}
 
Chow rings are intricate, important and hard to compute invariants of algebraic varieties. In this paper, we propose to study Chow rings by means of strictly semistable degenerations. Let $X_K$ be the smooth generic fibre of such a family and let $X$ be its special fibre.  

The first main result of this paper is that we construct in Definition \ref{dPrelogChow} the \emph{prelog Chow ring} $\Chowprelog (X)$ that admits (Theorem \ref{pSpecializationIntoChowPrelog}) a specialization homomorphism from the Chow ring of the generic fibre 
\[
\sigma : \Chow^* (X_K) \longrightarrow \Chowprelog (X).
\]

In Proposition \ref{pBothmer}, we derive a calculation scheme for $\Chowprelog (X)$ in terms of the Chow rings of the components of $X$ and their intersections. 

\

Our interest in the theory derives from the desire to extend Voisin's degeneration method \cite{Voi15} (as developed further by Colliot-Th\'{e}l\`{e}ne/Pirutka \cite{CT-P16}, Totaro \cite{To16}, Schreieder \cite{Schrei17,Schrei18} et al.) to more singular degenerations, for example toric degenerations used in the Gross-Siebert programme \cite{GS06, GS10, Gross11, GS11,GS11a, GS16,GS19}. 

The main idea on how to apply Voisin's method to strictly semistable degenerations can informally be described as follows: suppose $\XX \to \Delta$ is a degeneration of projective varieties $\XX_t$ over a small disk $\Delta$ centered at $0$ in $\C$ with coordinate $t$. Suppose $\XX_t$ is smooth for $t$ nonzero. Let $\XX^*\to \Delta^*$ be the induced family over the punctured disk obtained by removing the central fibre $\XX_0$. 
We would like to prove that a very general fibre $\XX_t$ is not retract rational (or weaker, not stably rational). Arguing by contradiction and assuming to the contrary that $\XX_t$ is stably rational, we obtain (cf. \cite[Proof of Thm. 1.1]{Voi15}) that, after replacing $t$ by $t^k$ for some positive integer $k$ and shrinking $\Delta$, there is a relative decomposition of the diagonal on $\XX^*\times_{\Delta^*}\XX^*\to \Delta^*$: there exists a section $\sigma\colon \Delta^* \to \XX^*$ and a relative cycle $\ZZ^* \subset \XX^*\times_{\Delta^*}\XX^*$ together with a relative divisor $\DD^* \subset \XX^*$ such that, for all $t\in \Delta^*$
\[
\Delta_{\XX_t} = \XX_t\times \sigma (t) + \ZZ^*_t \: \mathrm{in}\: \mathrm{CH}^* (\XX_t\times \XX_t)
\]
and $\ZZ^* \subset \DD^* \times_{\Delta*} \XX^*$. Closing everything up in $\XX\times_{\Delta}\XX$, and intersecting with $\XX_0\times\XX_0$, we can specialize this to obtain a decomposition of the diagonal on $\XX_0\times \XX_0$.

If the singularities of $\XX_0$ are sufficiently mild (e.g. only nodes, but more general classes of singularities are admissible), then a resolution $\tilde{\XX}_0$ inherits a decomposition of the diagonal. Now we can derive a contradiction to our initial assumption that $\XX_t$ is stably rational for $t$ very general since one can obstruct the existence of decompositions of the diagonal on nonsingular projective varieties by, for example, nonzero Brauer classes, and other unramified invariants. However, we can also try to bypass the necessity that $\XX_0$ have mild singularities if we are willing to obstruct \emph{directly the existence of a decomposition of the diagonal on $\XX_0\times \XX_0$ that arises as a ``limit" of decompositions of the diagonal on the fibres of $\XX^*\times_{\Delta^*}\XX^*\to \Delta^*$}. Here the qualifying relative clause is important: if the degeneration $\XX_0$ is sufficiently drastic, e.g. $\XX_0$ could be simple normal crossing with toric components, then decompositions of the diagonal may well exist, but still possibly none that arise as a limit in the way described above. 

The desire to single out decompositions of the diagonal for $\XX_0$ that would stand a chance to arise via such a limiting procedure naturally leads to the idea of endowing $\XX_0$ with its natural log structure $\XX_0^{\dagger}$ (cf. \cite[Chapter 3.2]{Gross11}) obtained by restricting the divisorial log structure for $(\XX, \XX_0)$ to $\XX_0$, and develop some kind of log Chow theory on $\XX_0^{\dagger}$ in which cycles carry some extra decoration that encodes information about the way they can arise as limits. Although a promising theory that may eventually lead to a realization of this hope is currently being developed by Barrott \cite{Bar18}, it is still at this point unclear to us if and how it can be used for our purposes, so we take a more pedestrian approach in this article, always keeping the envisaged geometric applications in view.

While we believe that the correct framework is ultimately likely to be the general theory of log structures in the sense of Deligne, Faltings, Fontaine, Illusie, Kato et al. as exposed in \cite{Og18}, our point of departure in Section \ref{sPrelogSNC} is the following simple observation: given a strictly semistable degeneration, cycle classes in $\Chow_* (\XX_0)$ that arise as specializations in fact come from cycle classes on the normalization of $\XX_0$ satisfying an obvious coherence/compatibility condition that we call the prelog condition following Nishinou and Siebert \cite{NiSi06,Ni15} (in the case of curves). Despite its simplicity, the idea is very effective in applications. We cast it in the appropriate algebraic structures in Section \ref{sPrelogSNC}. This takes the shape of the prelog Chow ring of a simple normal crossing scheme and its cousin the numerical prelog Chow ring, which is easier to handle computationally. We also discuss a relation between the prelog condition and the Friedman condition. 

In Section \ref{sSpecializations}, we treat specialization homomorphisms of strictly semistable degenerations into prelog Chow rings, and in Section \ref{sBaseChange} we deal with the problems that arise by the necessity to perform a ramified base change $t\mapsto t^k$ in Voisin's method outlined above. This leads to the concepts of saturated prelog Chow groups and their numerical counterparts. 

In Section \ref{sDecDiag} we define the concept of prelog decomposition of the diagonal and prove an analogue of Voisin's specialisation result in this case. 

In Section \ref{sExampleCubic} we compute the saturated prelog Chow group of a degeneration of cubic surfaces. We explicitly recover the 27 lines  as prelog cycles in the central fibre.

In Section \ref{sEllipticCurves}, we consider the case of a degeneration of a family of elliptic curves, realized as plane cubics, into a triangle of lines. We show that the central fibre cannot have a prelog decomposition of the diagonal, showing in particular again that a smooth elliptic curve does not have universally trivial Chow group of zero cycles. This is, although well-known, very reassuring because it shows that the concept of prelog decomposition of the diagonal is also a useful obstruction when dealing with degenerations all of whose components and mutual intersections of components are rational.

In Appendix \ref{sTropicalDegenerations} we explain some connections to the Gross-Siebert programme. 

Let us also mention that the concepts of prelog Chow rings and groups we introduce are eminently computable in even nontrivial examples such as degenerations of self-products of cubic threefolds, compare \cite{BBG19}.

Ideas related to the ones proposed in this article have been pursued in \cite{FM81, BGS94, NiShi19, NiOt19, Schrei19}, but we do not see how these could yield results for degenerations all of whose components and mutual intersections of components are rational.

\section*{Acknowledgement}

We thank Lawrence Barrott, Jean-Louis Colliot-Th\'el\`ene, Alessio Corti, Mathieu Florence, Mark Gross, Stefan Schreieder and Claire Voisin for valuable discussions surrounding the ideas presented in this paper. This project has received funding from the European Union's Horizon 2020 research and innovation programme under the Marie Sklodowska-Curie grant agreement No 746554 and has been supported by Dr.\ Max R\"ossler, the Walter Haefner Foundation and the ETH Z\"urich Foundation. The first author would like to acknowledge the stimulating environment and discussions of the 2019 AIM Workshop Rationality Problems in Algebraic Geometry that were helpful for coming to terms with various aspects of the present paper.
The first author was supported by the EPSRC New Horizons Grant EP/V047299/1.

\section{Prelog Chow rings of simple normal crossing schemes}\label{sPrelogSNC}
We work over the complex numbers $\C$ throughout. 

Let $X=\bigcup_{i\in I} X_i$ be a simple normal crossing (snc) scheme; here $I$ is some finite set, and all irreducible components $X_i$ are smooth varieties. Moreover, for a nonempty subset $J \subset I$, we denote by $X_J$ the intersection $\bigcap_{j\in J} X_j \subset X$. In this way, each $X_J$ is a smooth variety (possibly not connected). The irreducible components of $X_J$ then form the $(|J|-1)$-dimensional cells in the dual intersection complex of $X$. It is a regular cell complex in general, and simplicial if and only if all $X_J$ are irreducible. For nonempty subsets $J_1\subset J_2$ of $I$, we denote by
\[
\iota_{ J_2 > J_1}\colon X_{J_2} \hookrightarrow X_{J_1}
\]
and by
\[
\iota_J \colon X_J \hookrightarrow X
\]
the inclusions. Let
\[
\nu \colon X^{\nu} = \bigsqcup_{i\in I} X_i \to X
\]
be the normalization. 

\begin{definition}\label{dRingOfCompatibleClasses}
Denote by 
\[
R (X)=R \subset \Chow^* (X^{\nu}) = \bigoplus_{i\in I} \Chow^* (X_i)
\]
the following subring of the Chow ring of the normalization, which we call the \emph{ring of compatible classes}: elements in $R$ are tuples of classes $(\alpha_i)_{i\in I}$ with the property that for any two element subset $\{ j,k\} \subset I$
\[
\iota^*_{\{ j,k\} > \{j\}} (\alpha_j) = \iota^*_{\{ j,k\} > \{k\}} (\alpha_k).
\]
We call this property the \emph{prelog condition}. 
Furthermore, denote by 
\[
M = M(X)  = \Chow_* (X)
\]
the Chow group of $X$. 
\end{definition}

Notice that there is in general no well-defined intersection product on $M$. We will show, however, that one can turn $M$ into an $R$-module. We note that $R(X)$ agrees with Fulton-MacPherson's \emph{operational Chow ring}  \cite{FM81}.

\begin{definition}\label{dPairing}
Let $\alpha=(\alpha_i)_{i\in I}$ be an element in $R$ and let $Z$ be a prime cycle (=irreducible subvariety) on $X$. Let $J \subset I$ be the largest subset such that $Z \subset X_J$ and let $j\in J$ be arbitrary. 
Then we define 
\[
\langle \alpha , Z\rangle := \iota_{J, *} \left( \iota^*_{J > \{ j \}}(\alpha_j) . [Z] \right) .
\]
This is independent of the choice of $j$ since if $j'$ is another element of $J$, the class $\iota^*_{J > \{ j' \}}(\alpha_{j'})$ is the same as $\iota^*_{J > \{ j \}}(\alpha_j)$ since by the definition of $R$  we have that $\iota^*_{\{j, j'\} > \{ j\}}(\alpha_j) = \iota^*_{\{j, j'\} > \{ j'\}}(\alpha_{j'})$, so the definition is well-posed. 
If $Z$ is an arbitrary cycle on $X$, we define $\langle \alpha , Z\rangle$ by linearity. 
\end{definition}

\begin{proposition}\label{pIntersectionPairingPrelog}
If $Z_1$ and $Z_2$ are rationally equivalent cycles on $X$, then
\[
\langle \alpha , Z_1\rangle = \langle \alpha , Z_2\rangle .
\]
In particular, the pairing descends to rational equivalence on $X$ and makes $M=\Chow_* (X)$ into an $R$-module. The push forward map induced by the normalization
\[
\nu_*\mid_{R(X)} \colon R=R(X) \to M
\]
is an $R$-module homomorphism. Indeed, $\nu_* \colon \bigoplus_i \Chow_* (X_i) \to \Chow_* (X)$ is an $R$-module homomorphism.
\end{proposition}

\begin{proof}
The main point is the following consequence of the projection formula that gives a way to calculate the pairing $\langle \alpha , Z \rangle$ in a very flexible way: if $Z$ is a prime cycle contained in $X_J$ as in Definition \ref{dPairing}, then we can write
\begin{gather*}
\langle \alpha , Z\rangle = \iota_{J, *} \left( \iota^*_{J > \{ j \}}(\alpha_j) . [Z] \right) = \iota_{\{ j\} , *}\iota_{J> \{ j\}, *} \left( \iota^*_{J > \{ j \}}(\alpha_j) . [Z] \right) \\
=\iota_{\{ j\} , *} \left( \alpha_j . \iota_{J> \{ j\}, *}[Z] \right) .
\end{gather*}
Moreover, as already remarked in Definition \ref{dPairing}, here $j\in J$ is arbitrary and the result independent of it. Hence, for an arbitrary cycle $Z$ on $X$ and $\alpha\in R$, we can compute $\langle \alpha , Z\rangle$ as follows: first we write, in whichever way we like,  
\[
Z = \sum_{i\in I} Z_i
\]
where $Z_i$ is a cycle supported on $X_i$, then form the intersection products $\alpha_i . Z_i$ on $X_i$, push these forward to $X$ and sum to get the cycle $\langle \alpha , Z\rangle$. In particular, this makes it clear that if $Z_1$ and $Z_2$ are rationally equivalent on $X$, then $\langle \alpha , Z_1\rangle = \langle \alpha , Z_2\rangle$ as elements of $M$. Indeed, cycles of dimension $d$ rationally equivalent to zero are sums of cycles $T$ on $X$ that arise as follows: take an irreducible subvariety $Y \subset X$ of dimension $d+1$, its normalization $\nu_Y \colon Y^{\nu} \to Y$, and let $T$ be $\nu_{Y, *}$ of the divisor of zeros and poles of a rational function on $Y^{\nu}$. Now $Y$ is necessarily contained entirely within one of the irreducible components of $X$, $X_i$ say, and hence $T$ is rationally equivalent to zero on $X_i$. Hence $\langle \alpha , T\rangle =0$ since we have the flexibility to compute this entirely on $X_i$.

To show that $\nu_* \colon \bigoplus_i \Chow_* (X_i) \to M$ is an $R$-module homomorphism, we take two elements $\alpha = (\alpha_i)_{i\in I}\in R, \beta = (\beta_i)_{i\in I} \in \bigoplus_i \Chow_* (X_i)$ and represent $\beta_i$ by a cycle $Z_i$ on $X_i$. All we have to show then is that
\[
\nu_* (\alpha \beta ) = \langle \alpha , Z \rangle, \quad Z := \sum_{i\in I} \iota_{\{ i\}, *}(Z_i) .
\]
But this directly follows from the definition of $\nu_*$ and the way we can compute the pairing. 
\end{proof}

\begin{definition}\label{dPrelogChow}
We call the quotient $R(X)/(\mathrm{ker}\, \nu_*\mid_{R(X)})$ the prelog Chow ring of $X$, and denote it by $\Chowprelog (X)$ when graded by codimension and by $\Chow^{\rm prelog}_*(X)$ when graded by dimension. It is indeed naturally a ring since $\mathrm{ker}\,\nu_*\mid_{R(X)}$ is an ideal in $R(X)$. 
\end{definition}

\begin{proposition}\label{pKernelNu}
There is an exact sequence
\xycenter{
\bigoplus \Chow_*(X_{ij})  \ar[r]^\delta
&\bigoplus \Chow_*(X_{i})  \ar[r]^-{\nu_*} 
&\Chow_*(X) \ar[r]
& 0 
}
where for $z_{ij} \in  \Chow_*(X_{ij})$ with $i<j$ 
\begin{align*}
	\bigl(\delta(z_{ij})\bigr)_a 	
	&= \left\{ \begin{matrix}  
	     \ioLower{ij}{i}(z_{ij}) & \text{if $a=i$,} \\
	     -\ioLower{ij}{j}(z_{ij}) & \text{if $a=j$,} \\
	     0 & \text{otherwise.}
	\end{matrix} \right.
\end{align*}
\end{proposition}

\begin{proof}
We recall that by \cite[Ex. 1.8.1]{Ful98} if 
\[
\xymatrix{
Y' \ar[r]^j \ar[d]^q & Z' \ar[d]^{p}\\
Y \ar[r]^i & Z
}
\]
is a fibre square, with $i$ a closed embedding, $p$ proper, such that $p$ induces an isomorphism of $Z' - Y'$ onto $Z- Y$, then: there is an exact sequence
\[
\xymatrix{
\Chow_k Y' \ar[r]^{a\quad\quad} & \Chow_k Y \oplus \Chow_k Z' \ar[r]^{\quad\quad b} & \Chow_k Z \ar[r] & 0
}
\] 
where $a (\alpha ) = (q_* \alpha , \: -j_* \alpha), \: b(\alpha, \beta ) = i_* \alpha + p_* \beta$. 
Apply this inductively with $Y$ one component of an snc scheme and $Z'$ all remaining components. 
\end{proof}

\begin{example}\label{eP2F1}
Consider the snc scheme $X$ obtained by gluing $\P^2$ and the Hirzebruch surface $\mathbb{F}_1$ along a line $L$ in $\P^2$ identified with the $(-1)$-section of $\mathbb{F}_1$. Then $\Chow_1(X)$ is of rank 2 generated by the class of $L$ and the fibre class $F$ of $\mathbb{F}_1$.  $\Chow^{\rm prelog}_{1} X$ on the other hand is of rank 1 generated by the class $(L,F)$.

\end{example}

\begin{definition}\label{dFriedman}
Let $X$ be a snc scheme with at worst triple intersections. 
We say that $X$ satisfies the \emph{Friedman condition} if for every intersection $X_{ij} = X_i\cap X_j$ we have
\[
\NN_{X_{ij}/X_i}\otimes \NN_{X_{ij}/X_j} \otimes \OO (T) = \OO_{X_{ij}}.
\]
Here $T$ is the union of all triple intersections $X_{ijk}$ that are contained in $X_{ij}$. The Friedman condition is commonly also referred to as $d$-semistability.
\end{definition}

\begin{remark}\label{rFriedman}
By \cite[Def. 1.9 and Cor. 1.12]{Fried83} any $X$ that is smoothable with smooth total space has trivial infinitesimal normal bundle and in particular satisfies the Friedman condition. 
\end{remark}

The following Proposition describes a relation between the prelog condition, Friedman condition and Fulton's description of the kernel of $\nu_*$.

\begin{proposition}\label{pBothmer}
Let $X$ be an snc scheme that has at worst triple intersections and satisfies the Friedman condition. Then the following diagram commutes
 \xycenter{
  & 0 \ar[d] & 0 \ar[d] & \\
  & R(X) \ar[d] \ar[r] & \Chowprelog(X) \ar[d] \ar[r] & 0 \\
\bigoplus \Chow^*(X_{ij}) \ar[d]^-{\rho'} \ar[r]^\delta
&\bigoplus \Chow^*(X_{i})  \ar[r]^-{\nu_*} \ar[d]^-\rho
&\Chow_*(X) \ar[r]
& 0 \\
\bigoplus \Chow^*(X_{ijk}) \ar[r]^-{\delta'}
&\bigoplus \Chow^*(X_{ij})
 }
 
\newcommand{\II}{\{1,\dots,n\}}

Here the maps $\rho , \rho', \delta'$ are defined as follows, using the convention $a<b<c$, $i<j<k$:

\begin{align*}
	\bigl(\rho(z_{i})\bigr)_{ab} 	
	&= \left\{ \begin{matrix}  
	     \ioUpper{ab}{i}(z_{i}) & \text{if $i=a$} \\
	     -\ioUpper{ab}{i}(z_{i}) & \text{if $i=b$} \\
	     0 & \text{otherwise}
	\end{matrix} \right. \\
	\bigl(\rho'(z_{ij})\bigr)_{abc} 	
	&= \left\{ \begin{matrix}  
	     \ioUpper{abc}{ij}(z_{ij}) & \text{if $(i,j)=(a,b)$} \\
	     -\ioUpper{abc}{ij}(z_{ij}) & \text{if $(i,j)=(a,c)$} \\
	     \ioUpper{abc}{ij}(z_{ij}) & \text{if $(i,j)=(b,c)$} \\
	     0 & \text{otherwise}
	\end{matrix} \right. \\
	\bigl(\delta'(z_{ijk})\bigr)_{ab} 	
	&= \left\{ \begin{matrix}  
	     -\ioLower{ijk}{ab}(z_{ijk}) & \text{if $(a,b)=(i,j)$} \\
	     \ioLower{ijk}{ab}(z_{ijk}) & \text{if $(a,b)=(i,k)$} \\
	     -\ioLower{ijk}{ab}(z_{ijk}) & \text{if $(a,b)=(j,k)$} \\
	     0 & \text{otherwise}
	\end{matrix} \right. 
\end{align*}
 Notice that being in the kernel of $\rho$ amounts to the prelog condition.
 \end{proposition}

\begin{proof}
It remains to be seen that the lower square commutes. For this we want to prove
\[
	\bigl((\delta' \circ \rho' )(z_{ij})\bigr)_{ab} = \bigl((\rho \circ \delta)(z_{ij})\bigr)_{ab}. 
\]
There are three cases:

{\bf Case 1}: $|\{i,j\} \cap \{a,b\}| = 0$. In this case both sides of the equation are $0$.

{\bf Case 2}: $|\{i,j\} \cap \{a,b\}| = 1$. In this case we can assume $\{i,j\} \cup \{a,b\} = \{i,j,k\}$. Depending on the relative size of the indices involved there are a number of cases to consider. We check only the case $a=j, b=k$, the others are similar. The left hand side then is
\[
             \bigl((\delta' \circ \rho' )(z_{ij})\bigr)_{jk}
             = \bigl( \delta' ( \ioUpper{ijk}{ij}(z_{ij}) \pm \dots)\bigr) _{jk}
             = -\ioLower{ijk}{jk} \ioUpper{ijk}{ij}(z_{ij}).
\]
Similarly the right hand side is
\begin{align*}
             \bigl((\rho \circ \delta)(z_{ij})\bigr)_{jk}
             &= \Bigl(\rho \bigl(\ioLower{ij}{i}(z_{ij})-\ioLower{ij}{j}(z_{ij})\bigr)\Bigr)_{jk}\\
             &= 0 - \Bigl(\rho \bigl( \ioLower{ij}{j}(z_{ij})\bigr)\Bigr)_{jk}\\
            &= - \ioUpper{jk}{j} \ioLower{ij}{j}(z_{ij}).
\end{align*}
So our claim is just the commutativity of the diagram
\xycenter{
	\Chow^*(X_{ij}) \ar[r]^{\iota_*} \ar[d]^{\iota^*} & \Chow^*(X_j) \ar[d]^{\iota^*}\\
	\Chow^*(X_{ijk}) \ar[r]^{\iota_*} & \Chow^*(X_{jk}).
}

{\bf Case 3}: $\{i,j\} = \{a,b\}$. Here we get on the left hand side
\begin{align*} 
	\bigl((\delta' \circ \rho' )(z_{ij})\bigr)_{ij}
             &= \left( \delta' \left ( \sum_k \pm \ioUpper{ijk}{ij}(z_{ij}) \right)\right) _{ij} \\
             &= - \sum_k   \ioLower{ijk}{ij}(z_{ij}) \ioUpper{ijk}{ij}(z_{ij})
\end{align*}
since by our sign convention $\rho'$ and $\delta'$ always induce opposite signs in this situation.
On the right hand side we use the Friedman relation:
\begin{align*}
             \bigl((\rho \circ \delta )(z_{ij})\bigr)_{ij}
             &= \Bigl(\rho \bigl(\ioLower{ij}{i}(z_{ij})-\ioLower{ij}{j}(z_{ij})\bigr)\Bigr)_{ij}\\
            &= \ioUpper{ij}{i} \ioLower{ij}{i}(z_{ij}) - \bigl(-\ioUpper{ij}{j} \ioLower{ij}{j}(z_{ij})\bigr)\\
            &= N_{X_{ij}/X_i} \cdot z_{ij} + N_{X_{ij}/X_j} \cdot z_{ij} \\
           &= \left( -\sum_k \ioLower{ijk}{ij}(X_{ijk})  \right) \cdot z_{ij} \\
           & = -\sum_k \ioLower{ijk}{ij} \ioUpper{ijk}{ij} (z_{ij}). 
\end{align*}
\end{proof}

One drawback of $\Chowprelog (X)$ is that it is rather hard to compute with in examples, for instance it can be very far from being finitely generated. Instead, we would like to have an object constructed using numerical equivalence that receives at the very least an arrow from $\Chowprelog (X)$. 

\begin{definition}\label{dBothmer}
Let $X$ be an snc variety. Then we define $R_{\mathrm{num}}(X)$ and $ \Chowprelognum (X)$ via the following  diagram induced by the diagram in Proposition \ref{pBothmer}: 
 \xycenter{
  & 0 \ar[d] & 0 \ar[d] & \\
  & R_{\mathrm{num}}(X) \ar[d] \ar[r] & \Chowprelognum (X) \ar[d] \ar[r] & 0 \\
\bigoplus \Chownum^*(X_{ij})  \ar[r]^\delta
&\bigoplus \Chownum^*(X_{i})  \ar[r]^-{} \ar[d]^-\rho
&\coker (\delta ) \ar[r]
&  0 \\
&\bigoplus \Chownum^*(X_{ij}) & 
 }
 
\newcommand{\II}{\{1,\dots,n\}}

Notice that being in the kernel of $\rho$ amounts to the prelog condition. 
  
 \end{definition}


\begin{proposition}\label{pHomoIntoNumprelog}
The natural projection map
\[
\varpi \colon R(X) \to R_{\mathrm{num}}(X)
\]
maps classes in the kernel of $\nu_*\mid_{R(X)}$ into $\im\, \delta \cap R_{\mathrm{num}}(X)$. Hence we obtain an induced homomorphism
\[
\bar{\varpi}\colon \Chowprelog (X) \to \Chowprelognum (X).
\]
\end{proposition}

\begin{proof}
This is clear by construction.
\end{proof}

\section{Specialization homomorphisms into prelog Chow rings}\label{sSpecializations}

We start by recalling some facts about specialization homomorphisms, following the original paper \cite{Ful75} or the standard reference \cite{Ful98}. Let
\[
\pi \colon \XX \to C
\]
be a flat morphism from a variety $\XX$ to a nonsingular curve $C$. We fix a distinguished point $t_0 \in C$, and denote by $X$ the scheme-theoretic fibre $\XX_{t_0} = \pi^{-1}(t_0)$. Let $i \colon X \to \XX$ be the inclusion. As in \cite[\S 4.1, p. 161]{Ful75} (or \cite[Chapter 2.6]{Ful98}) one can then define a ``Gysin homomorphism"
\[
i^* \colon \Chow_{k} (\XX ) \to \Chow_{k-1} (X)
\]
by defining the map $i^* \colon Z_k (\XX ) \to Z_{k-1} X$ and checking that this descends to rational equivalence; on the level of cycles, if an irreducible subvariety $V$ of $\XX$ satisfies $V \subset X$, one defines $i^* (V) =0$, and otherwise as $i^* (V) = V_{t_0}$, where $V_{t_0}$ is the cycle associated to the zero scheme on $V$ of a regular function defining $X$ inside $\XX$ in a neighborhood of $X$ (notice that $X$ is a principal Cartier divisor in $\XX$ after possibly shrinking $C$). In the latter case, the class of $V_{t_0}$, well-defined as an element in $\Chow_{k-1} (|X|\cap V)$, is then the intersection $X\cdot V$ of $V$ with the Cartier divisor $X$; see also \cite[Chapter 2.3 ff.]{Ful98} for further information on this construction.

From now on we will want to work more locally on the base, hence assume that $C= \mathrm{Spec}\, R$ is a \emph{curve trait}, by which we mean that $R$ is a discrete valuation ring that is the local ring of a point on a nonsingular curve, or a completion of such a ring. By \cite[\S 4.4]{Ful75}, all what was said above remains valid in this set-up. Let $\C=R/\mathfrak{m}$ be the residue field of $R$, and $K$ the quotient field. We then have the special fibre $X_{\C}=X$ and generic fibre $X_K$ with inclusions $i\colon X \to \XX$ and $j \colon X_K\to \XX$. 

By \cite[Prop. in \S 1.9]{Ful75}, there is an exact sequence
\[
\xymatrix{
\Chow_{p+1} X \ar[r]^{i_*} & \Chow_{p+1} \XX \ar[r]^{j^*} & \Chow_{p} X_K \ar[r] & 0
}
\]
and $i^*i_* =0$, so one gets that there is a unique map $\sigma=\sigma_{\XX}$ forming 
\[
\xymatrix{
  & \Chow_p X_K \ar[dd]^{\sigma} \\
\Chow_{p+1} \XX \ar[ru]^{j^*} \ar[rd]^{i^*} & \\
 & \Chow_{p} X
}
\]
This $\sigma$ is called the \emph{specialization homomorphism}.

\begin{definition}\label{dStrictlySemistable}
The flat morphism $\pi \colon \XX \to C$ with $C$ a curve trait, or a nonsingular curve, with marked point $t_0\in C$, is called a strictly semistable degeneration if $X=\pi^{-1}(t_0)$ is reduced and simple normal crossing and $\XX$ is a regular scheme. 
\end{definition}

The advantage of $\XX$ being regular is that then all the components $X_i$ of $X$ are Cartier divisors in $\XX$, and we can form intersections with each of them separately. 

\begin{theorem}\label{pSpecializationIntoChowPrelog}
If $\pi \colon \XX \to C$ is a strictly semistable degeneration, $\sigma$ takes values in $\Chowprelog (X)$. 
\end{theorem}

\begin{proof}
The technical heart of the proof is that by \cite[Chapter 2]{Ful98}, given a $k$-cycle $\alpha$ and Cartier divisor $D$ on some algebraic scheme, one can construct an intersection class $D\cdot \alpha \in \Chow_{k-1} (|D|\cap |\alpha |)$, satisfying various natural properties. Let $\alpha =[Y]$ be the class of a subvariety $Y$. If $Y \not\subset |D|$, then $Y$ restricts to a well-defined Cartier divisor on $D$ whose associated Weil divisor is defined to be $D\cdot \alpha$; if $Y\subset |D|$, then one defines $D\cdot \alpha$ as the linear equivalence class of any Weil divisor associated to the restriction of the line bundle $\OO (D)$ to $Y$. 

Let now $V$ be an irreducible subvariety of $\XX$ not contained in $X$. We have $X=\bigcup_i X_i$, and all components $X_i$ are Cartier. By \cite[Prop. 2.3 b)]{Ful98}, one has in $\Chow_* (|X|\cap V)$
\[
X \cdot V= (\sum_i X_i) \cdot V = \sum_i X_i \cdot V. 
\]
Hence defining $\alpha_i = X_i \cdot V$ (viewed as classes in $\Chow_* (X_i)$), we will have proved the Proposition once we show that $(\alpha_i)$ is in the ring $R(X)$ of compatible classes. This follows from the important commutativity property of the pairing between Cartier divisors and cycles on any algebraic scheme: if $D, D'$ are Cartier divisors and $\beta$ a cycle, then
\[
D \cdot (D' \cdot \beta) = D' \cdot (D\cdot \beta )
\]
as classes in $\Chow_* (|D|\cap |D'|\cap |\beta|)$ by \cite[Cor. 2.4.2 of Thm. 2.4]{Ful98}. We now only have to unravel that this boils down to the property we seek to prove: indeed, 
\[
X_j \cdot \alpha_i = X_i \cdot \alpha_j 
\]
and it follows from the definitions that we can compute $X_j \cdot \alpha_i$ as follows: $X_j$ restricts to a well-defined Cartier divisor on the subvariety $X_i$, which is nothing but the Cartier divisor associated to $X_i\cap X_j$; intersecting that Cartier divisor with the cycle $\alpha_i$ on $X_i$ gives $X_j \cdot \alpha_i$; analogously, for $X_i \cdot \alpha_j $ with the roles of $i$ and $j$ interchanged. Hence $(\alpha_i)$ is in $R(X)$. 
\end{proof}

\begin{remark}\label{rWhySmoothTotalSpace}
The assumption that the total space $\XX$ be nonsingular in Theorem \ref{pSpecializationIntoChowPrelog} is essential, and it is useful to keep the following example in mind: take a degeneration of a family of plane conics into two lines, and then consider the product family of this with itself. The central fibre of the product family is a union of four irreducible components all of which are isomorphic to $\P^1\times \P^1$ glued together along lines of the rulings as in Figure \ref{PrelogIllustration2} on the left-hand side. The cycle indicated in green is the specialization of the diagonal, and the red cycles 1, 2, 3 are all rationally equivalent, but the intersections with the green cycle are \emph{not} rationally equivalent. Moreover, the central fiber in \textcircled{1} is not snc, but in fact, one can nevertheless define the prelog Chow ring as in Section \ref{sPrelogSNC}: one only needs the intersections of every two irreducible components to be smooth for this, and all results of that Section, in particular, Proposition \ref{pIntersectionPairingPrelog} hold with identical proofs if in addition the intersection of every subset of the irreducible components is still smooth, which is the case here. So what is wrong? The point is that the specialization of the diagonal is \emph{not} in $\Chowprelog (X)$ here. For example, the green cycle in \textcircled{1} on the upper left hand $\P^1\times\P^1$ intersects the mutual intersection of the two upper $\P^1\times\P^1$'s in a point, but there is no cycle on the upper right-hand $\P^1\times\P^1$ to match this to satisfy the pre-log condition. 

However, suppose we blow up one of the four components in the total space (this component is a non-Cartier divisor in the total space), thus desingularizing the total space and getting a new central fibre as in \textcircled{2}. Then the specialization of the diagonal will look like the green cycle and \emph{is} in $\Chowprelog (X)$. The red cycle in 1 here is equivalent to the red cycle in 2 and, if we try to move it across to the upper left-hand irreducible component, we see that the cycle 1 is equivalent to the dashed cycle 3, with multiplicities assigned to the irreducible components of the cycle as indicated. The intersection with the green cycle remains constant in accordance with Proposition \ref{pIntersectionPairingPrelog}: first it is $0$, then it is $+1-1=0$ again.

\begin{figure}[h]
	\centering
	\includegraphics[scale=0.45]{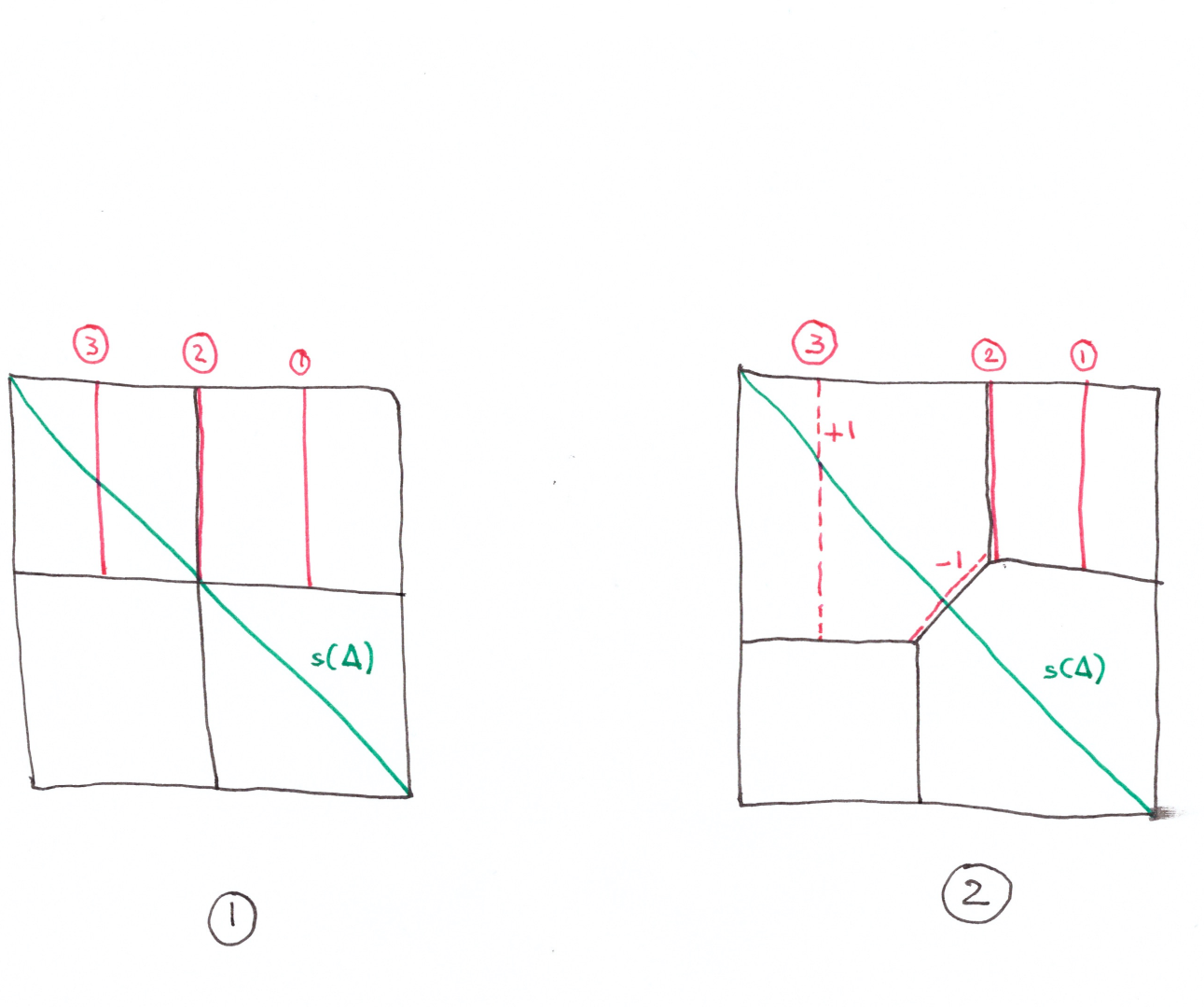}
	\caption{}\label{PrelogIllustration2}
\end{figure}
\end{remark}

\begin{remark}
One cannot expect $\sigma$ to be injective, at least not for point classes ($p=0$), as the degeneration of an elliptic curve into a cycle of rational curves shows. 
\end{remark}

\begin{example}\label{edegnormalcone}
Returning to the setup of Example \ref{eP2F1},
let $L$ be a line in $\P^2$ and consider the strictly semistable family $\XX \to \A^1$ obtained by blowing up $L\times\{0\}$ in $\P^2\times\A^1$ (degeneration to the normal cone). Restricting to a curve trait, denote by $X_K\cong\P^2$ the generic fibre and by $X$ the special fibre. $X$ has two components, $\P^2$ and the Hirzebruch surface $\mathbb{F}_1$, glued by identifying $L$ with the $(-1)$-section of $\mathbb{F}_1$. Then the image under $\sigma$ of the hyperplane class is the generator $(L,F)$ of $\Chow^{\rm prelog}_{1} X$ and $\sigma$ is an isomorphism.

\end{example}

\section{Ramified base change and saturated prelog Chow groups}\label{sBaseChange}

We keep working in the setup of the previous Section, and consider a strictly semistable degeneration $\pi \colon \XX \to C$. Suppose that $\beta \colon C' \to C$ is some cover of smooth curves or curve traits, in general ramified at the distinguished point $t_0\in C$. Suppose $t_0'$ is a distinguished point in $C'$ mapping to $t_0$ under $\beta$. We consider the fibre product diagram
\[
\xymatrix{
\XX' = \XX\times_C C'\ar[d]^{\pi'}  \ar[r]^{\quad\quad\beta'} & \XX \ar[d]^{\pi}\\
C' \ar[r]^{\beta} & C 
}
\]
Then $\XX'$ will in general be singular. However, we can still prove that the specialization homomorphism $\sigma_{\XX'}$ will take values, modulo torsion, in a group that is very similar to the prelog Chow ring of $X=\XX_{t_0}$.

\begin{definition}\label{dSaturatedPrelogChowGroup}
Let $X$ be a simple normal crossing variety. We define \emph{the saturated prelog Chow group} $\Chowprelogsat_{,*} (X)$ as the saturation of the image of $\Chowprelog (X)$ in $\coker (\delta )/(torsion)$ with $\delta$ as in the diagram of Proposition \ref{pBothmer}.

\medskip

We define the \emph{saturated numerical prelog Chow group} $\Numsatprelog^* (X)$ as the saturation of the image of $ \Chowprelognum (X)$ in the lattice $\coker (\delta )/(torsion)$ with $\delta$ as in the diagram in Definition \ref{dBothmer}.
\end{definition}

\begin{proposition}\label{pSpecializationIntoSat}
With the notation introduced above, the specialization homomorphism $\sigma_{\XX'}$ associated to $\pi'\colon \XX'\to C'$ takes values in the group $\Chowprelogsat_{,*} (X)$ after we mod out torsion from $\Chow_* (X)$. 
\end{proposition}

\begin{proof}
The punchline of the argument is very similar to that used in the proof of Theorem \ref{pSpecializationIntoChowPrelog}. The irreducible components $X_i$ of $X$, viewed as the fibre of the family $\pi'\colon \XX'\to C'$ over $t_0'$, are $\Q$-Cartier, thus there is an integer $N$ such that each $D_i:= NX_i$ is Cartier. This is so because the $X_i$ are Cartier in $\XX$, and local equations of $X_i$ in $\XX$ pull back, under $\beta'$, to local equations of $X_i$, with some multiplicity, inside $\XX'$. Hence it makes sense, given an irreducible subvariety $V$ of $\XX'$ not contained in $X$, to form the intersection products $\alpha_i = D_i \cdot V$ and view them as classes in $\Chow_* (X_i)$. In $\Chow_* (X_i)\otimes_{\Z}\Q$, we can then define the classes $\gamma_i := (1/N) \alpha_i$. Clearly, we have again, as in the proof of Theorem \ref{pSpecializationIntoChowPrelog}, 
\[
NX \cdot V= (\sum_i D_i) \cdot V = \sum_i D_i \cdot V. 
\]
Thus $\sigma_{\XX'} (V)$ and $\nu_* (\gamma_i)$ define the same class in $\Chow_* (X)/(\mathrm{torsion})$. Hence it remains to show that $(\gamma_i)$ is in $R(X)^{\Q}$. We use once more the commutativity property for intersections of cycles with Cartier divisors:
\[
D_i\cdot (D_j \cdot V) = D_j \cdot (D_i \cdot V)
\]
holds in $\Chow_* (|D_i|\cap |D_j|\cap |V|)$ by \cite[Cor. 2.4.2 of Thm. 2.4]{Ful98}. We need to convince ourselves that this indeed implies that $(\gamma_i)$ is in $R(X)^{\Q}$. Immediately from the definition, we see that we can compute $D_i \cdot (D_j \cdot V)$ as follows: consider $\alpha_j$ as a class on $X_j$; $D_i$ then restricts to a well-defined Cartier divisor on $X_j$, namely the Cartier divisor associated to $X_i\cap X_j$ (without multiplicity). Then $D_i\cdot (D_j \cdot V)$ is just the intersection of $X_i\cap X_j$ and $\alpha_j$, taken on $X_j$. The same for $D_j \cdot (D_i \cdot V)$ with $i$ and $j$ interchanged. In other words, $N\gamma_i$ and $N\gamma_j$ pull-back to the same class in $\Chow_* (X_i\cap X_j)$, hence the assertion.
\end{proof}


\begin{proposition}\label{pHomoIntoNumsatprelog}
There exists a natural homomorphism of modules
\[
 \Chowprelogsat_{,*}(X) \to \Numsatprelog_{,*} (X).
\]
\end{proposition}

\begin{proof}
Clear by construction.
\end{proof}

\section{Prelog decompositions of the diagonal}\label{sDecDiag}

We recall some facts concerning the degeneration method in a way that will be suitable to develop our particular view point on it. We will follow \cite{Voi15} and \cite{CT-P16}. 

\begin{definition}\label{dDecomposition}
Let $V$ be a smooth projective variety of dimension $d$ (over any field $K$). We say that $V$ has a \emph{decomposition of the diagonal} if one can write
\[
[\Delta_V] = [V\times p] + [Z] \quad \mathrm{in}\: \Chow_d (V\times V)
\]
where $p$ is a zero-cycle of degree $1$ on $V$ and $Z \subset V\times V$ is a cycle that is contained in $D\times V$ for some codimension $1$ subvariety $D\subset V$.
\end{definition}

This is equivalent to $V$ being universally $\mathrm{CH}_0$-trivial (meaning the degree homomorphism $\deg\colon \mathrm{CH}_0 (V_{K'}) \to \Z$ is an isomorphism for any overfield $K'\supset K$) by \cite[Prop. 1.4]{CT-P16}. If it holds, it holds with $p$ replaced by any other zero-cycle of degree $1$, in particular for $p$ a $K$-rational point if $V$ has any. If $V$ is smooth and projective and stably rational over $K$ (or more generally retract rational), then $V$ has a decomposition of the diagonal by \cite{ACTP17}, see also \cite[Lemm. 1.5]{CT-P16}.

We now consider a degeneration 
\[
\pi_{\VV} \colon \VV \to C
\]
over a curve $C$ with distinguished point $t_0$ and special fibre $V=\VV_{t_0}$. We usually assume a general fibre of $\pi_{\VV}$ to be rationally connected and smooth (in particular, all points will be rationally equivalent on it), since otherwise the question is not interesting. Let $K$ be the function field of $C$. Suppose that a very general fibre of $\pi_{\VV}$ is even stably rational, then this is also true for the geometric generic fibre $\VV_{\bar{K}}$. Indeed, for $b$ outside of a countable union of proper subvarieties of $C$ we have a diagram
\[
\xymatrix{
\VV_{\bar{K}}\ar[d] \ar[r]^{\simeq}_j & \VV_b \ar[d]\\
\mathrm{Spec}\, \bar{K} \ar[r]^{\simeq}_i & \mathrm{Spec}\, \C
}
\]
for some isomorphism $i$ and an isomorphism $j$ of schemes (see \cite[Lemma 2.1]{Vi13}). In particular, $\VV_{\bar{K}}$ then has a decomposition of the diagonal and is universally $\mathrm{CH}_0$-trivial. 

Now suppose that we want to prove that a very general fibre of $\pi_{\VV}$ is not stably rational. We would then assume the contrary, arguing by contradiction, and the classical degeneration method \cite[Thm. 1.14]{CT-P16} would then proceed as follows: \emph{assume that the special fibre $V$ is sufficiently mildly singular, in particular integral with a $\mathrm{CH}_0$-universally trivial resolution of singularities $f\colon \tilde{V}\to V$ (meaning $f_* \colon \mathrm{CH}_0 (\tilde{V}_{K'}) \to \mathrm{CH}_0  (V_{K'})$ is an isomorphism for all overfields $K'\supset K$)}. Then the fact that $\VV_{\bar{K}}$ is universally $\mathrm{CH}_0$-trivial would imply that $\tilde{V}$ is universally $\mathrm{CH}_0$-trivial as well. $\tilde{V}$ being smooth, we can now use various obstructions, such as nonzero Brauer classes, to show that $\tilde{V}$ is in fact not universally $\mathrm{CH}_0$-trivial and get a contradiction.


\begin{definition}\label{dPrelogCycle}
A \emph{prelog $d$-cycle $Z$} on some simple normal crossing scheme $X =\bigcup_i X_i$ is a tuple of $d$-cycles $(Z_i)$ with support $|Z_i|$ on the normalization $X^{\nu}$ such that for every $X_{ij}=X_i\cap X_j$ we have
\begin{enumerate}
\item[(a)]
$[Z_i] .[X_{ij}] = [Z_j]. [X_{ij}]$ in $\Chow_* (X_{ij})$. 
\item[(b)]
$|Z_i|\cap X_{ij} = |Z_j|\cap X_{ij}$, set-theoretically.
\end{enumerate}
Here (a) means that $([Z_i])\in R(X)$. 
\end{definition}

A prelog cycle determines a prelog cycle class in $\Chowprelog(X)$.

\begin{definition}\label{dPrelogDecompositionDiagonal}
Let $\pi_{\VV}\colon \VV \to C$ be a strictly semistable degeneration over a smooth curve $C$ with distinguished point $t_0$. Let $\pi_{\XX}\colon \XX \to C$ be a strictly semistable family fitting into a diagram
\[
\xymatrix{
\XX \ar[r]^{\varrho} \ar[rd]^{\pi_{\XX}} & \VV \times_C \VV \ar[d] \\
 & C 
}
\]
where the map $\varrho$ is a birational morphism that is an isomorphism outside the central fibres.

 Let $X_i$, $i\in I$, be the irreducible components of $X$, and $X_{ij}=X_i\cap X_j$ their mutual intersections for $i\neq j$. 

We say that $X$ \emph{has a prelog decomposition of the diagonal relative to the given family $\XX\to C$} if there exists a ramified cover $C'\to C$, and a point $t_0'\mapsto t_0$, such that the following hold: let $\sigma_{\XX'}$ be the specialisation homomorphism for the family $\XX'\to C'$ induced by $\XX\to C$
\[
\sigma_{\XX'} \colon \Chow_* (\VV_L\times_L \VV_L) \to \Chow_* (\XX_{t_0'}).
\]
Define 
\[
\zeta = \sigma_{\XX'}([\Delta_{\VV_L}]) - \sigma_{\XX'}([\VV_L \times \{p_L\} ]) 
\]
for $p_L$ an $L$-rational point of $\VV_L$. We use the identification $X = \XX'_{t_0'}$. Then we require that there exist cycles $A_i$ on $X_i$ such that
\begin{enumerate}
\item[(1)]
$(A_i)$ is a prelog cycle on $X$.
\item[(2)]
Let $\nu\colon \bigcup_{i\in I} X_i \to X$ be the normalisation. Then the \emph{cycle}
\[
\sum_{i\in I} \nu_* A_i
\]
is divisible by $r\in \mathbb{N}_{>0}$ as a \emph{cycle} on $X$ (not a cycle class!) and the class of the cycle
\[
\frac{1}{r} \sum_{i\in I} \nu_* A_i
\]
is equal to $\zeta$. 
\item[(3)]
None of the $A_i$ dominate any component of $V$ under the morphism $\mathrm{pr}_1\circ \varrho : X \to V\times V \to V$. 
\end{enumerate}
\end{definition}

Notice that a family $\XX \to C$ satisfying the assumptions of the above Definition can be obtained according to \cite[Prop 2.1]{Har01} by some succession of blow-ups of components of the central fibre of $\VV \times_C \VV \to C$ that are not Cartier in the total space. 

\begin{theorem}\label{tPrelogDecompositionInherited}
Given a strictly semistable degeneration $\VV\to C$ such that $\VV_{\bar{K}}$ has a (Chow-theoretic) decomposition of the diagonal, and given any strictly semistable modification $\XX \to C$ of the product family $\VV\times_C \VV \to C$, the central fibre $X$ has a prelog decomposition of the diagonal relative to $\XX\to C$.
\end{theorem}

\begin{proof}
Since $\VV_{\bar{K}}$ has a (Chow-theoretic) decomposition of the diagonal, there is a finite extension $L\supset K$ corresponding to a ramified cover $C'\to C$ such that $\VV_{L}:=\VV_K\times_KL$ has a decomposition of the diagonal. This means that there is a cycle $Z_L$ on $\VV_L\times \VV_L$ and an $L$-rational point $p_L$ of $\VV_L$ such that 
\[
[\Delta_{\VV_L}] = [\VV_L \times \{p_L\}] + [Z_L]
\]
and $Z_L$ does not dominate $\VV_L$ via the first projection.

Consider $\pi_{\XX'}\colon \XX'\to C'$ and a point $t_0'$ mapping to $t_0$. Again identify $\XX'_{t_0'}$ with $X$. The cycle $Z_L$ corresponds to a relative cycle $\ZZ_U$ over some open subset $U$ in $C'$. The image of the specialisation map applied to $Z_L$ is obtained as follows: take the closure $\ZZ$ of $\ZZ_U$ in $\XX'$ and intersect this with $X$, which is a Cartier divisor since 
\[
X =\pi_{\XX'}^{-1}(t_0'). 
\]
This gives a well-defined cycle class in $\Chow_* (X)$ which is $\sigma_{\XX'}(Z_L)$. Since $\XX'\to C'$ is a base change of a strictly semistable family $\XX\to C$, in particular, $\XX$ is smooth, all components $X_i$ of $X$ are $\mathbb{Q}$-Cartier on $\XX'$; hence we can find a natural number $r$ such that every $rX_i$ is Cartier on $\XX'$. More precisely, if the map of germs of pointed curves $(C', t_0') \to (C, t_0)$ is locally given by $t\mapsto t^s$, then locally around a general point on the intersection of two components of $X$, the family $\XX$ is given by $xy -t =0$, and the family $\XX'$ by $xy-t^s=0$, and we can choose $r=s$.

 We can define cycles $A_i$ on $X_i$ by intersecting $rX_i$ with $\ZZ$. The properties (a) and (b) of Definition \ref{dPrelogCycle} are then satisfied. For (b) this is clear by construction. For (a) note that we have 
 \[
 \ZZ . (rX_i).(rX_j) = \ZZ . (rX_j).(rX_i)
 \]
 as cycle classes on $\ZZ\cap X_i\cap X_j$. Let us interpret the left hand side: firstly $\ZZ.(rX_i) =A_i$, a cycle class supported on $X_i$. To intersect this further with $rX_j$, one has to restrict the Cartier divisor $rX_j$ to $A_i$. To do this one can first restrict $rX_j$ to $X_i$ to get a Cartier divisor on $X_i$ which one then restricts further to $A_i$. Now $rX_j$ restricted to $X_i$ is nothing but the Cartier divisor $X_{ij}$ on $X_i$: from the local equations above, if $X_i$ is locally given by $x=t=0$ and $X_j$ by $y=t=0$, and the Cartier divisor $rX_j$ by $y=0$ on $\XX'=(xy-t^r=0)$, then restricting the equation $y=0$ to $x=t=0$ gives the Cartier divisor $X_{ij}=(x=y=t=0)$ on $X_i=(x=t=0)$. Thus $\ZZ . (rX_i).(rX_j)  = A_i .[X_{ij}]$ and analogously $\ZZ . (rX_j).(rX_i) = A_j. [X_{ij}]$. This shows (a). Hence Definition \ref{dPrelogDecompositionDiagonal}(1) holds.
 
  The cycle 
\[
\sum_{i\in I} \nu_* A_i
\]
is by construction equal \emph{as a cycle} to $r$ times the cycle obtained by intersecting $\ZZ$ with $X$ (since they are equal as cycle classes supported on $|\ZZ|\cap |X|$ by Fulton's refined intersection theory). Hence Definition \ref{dPrelogDecompositionDiagonal}(2) holds. 

For the property Definition \ref{dPrelogDecompositionDiagonal}(3), notice that there exists a relative divisor $\DD_U$ in $\VV'$ over $U$, where $\VV'\to C'$ is obtained via base change from $\VV\to C$, such that $\ZZ_U$ is contained in $\DD \times_{C'} \VV'$. Then $\ZZ \cap X$ is contained in the intersection of the closure of $\DD \times_{C'} \VV'$ in $\XX'$ with $X$, which maps to the intersection of the closure of $\DD \times_{C'} \VV'$ in $\VV'\times_{C'}\VV'$ with $V\times V$ under $\varrho$. 
\end{proof}

\begin{remark}\label{rAI}
Notice that the class
\[
\left[ \frac{1}{r} \sum_{i\in I} \nu_* A_i \right]
\]
is an element of $\Chowprelogsat (X)$. Moreover, the classes in $\Chowprelogsat (X)$ arising in this way from prelog cycles $(A_i)$ satisfying (1), (2), (3) in Definition \ref{dPrelogDecompositionDiagonal} form a subgroup $\Chowprelognotdominant (X)$ of $\Chowprelogsat (X)$. With this $X$ has a prelog decomposition of the diagonal relative to the given family $\XX\to C$ if $\zeta \in \Chowprelognotdominant (X)$, where we defined $\zeta = \sigma_{\XX'}([\Delta_{\VV_L}]) - \sigma_{\XX'}([\VV_L \times \{p_L\} ]) $.
\end{remark}

\begin{remark}\label{rStrengthening}
In the situation of Theorem \ref{tPrelogDecompositionInherited} we can also assume more strongly that the central fibre $X$ has a prelog decomposition of the diagonal relative to $\XX\to C$ satisfying the following additional condition: every connected component of the cycle $\sum_i \nu_* A_i$ of Definition \ref{dPrelogDecompositionDiagonal} satisfies Condition (1) of that Definition separately. 
\end{remark}

\section{Saturated prelog Chow groups of degenerations of cubic surfaces}\label{sExampleCubic}

In this section we consider a degeneration of a smooth cubic surface $S$ into three planes. More precisely let $f=0$ be an equation of $S$ and let
\xycenter{
	\YY  \ar[d] \ar[r] & \P^3 \times \A^1 \ar[d] \\
	\A^1 \ar@{=}[r] & \A^1
}
be the degeneration defined by
\[
	xyz - tf = 0.
\]
Here $x,y,z,w$ are coordinates of $\P^3$ and $t$ is the coordinate of $\A^1$. The central
fiber $Y$ of $\YY$ consists of three coordinate planes. The singularities of $\YY$ lie
on the intersection of two of these planes with $f=0$. We assume $f$ to be generic enough, such that these are three distinct points on each line, none of them in the origin. Figure \ref{fThreePlanes}
depicts this situation.

\begin{figure}[h]
	\centering
	\includegraphics[scale=0.45]{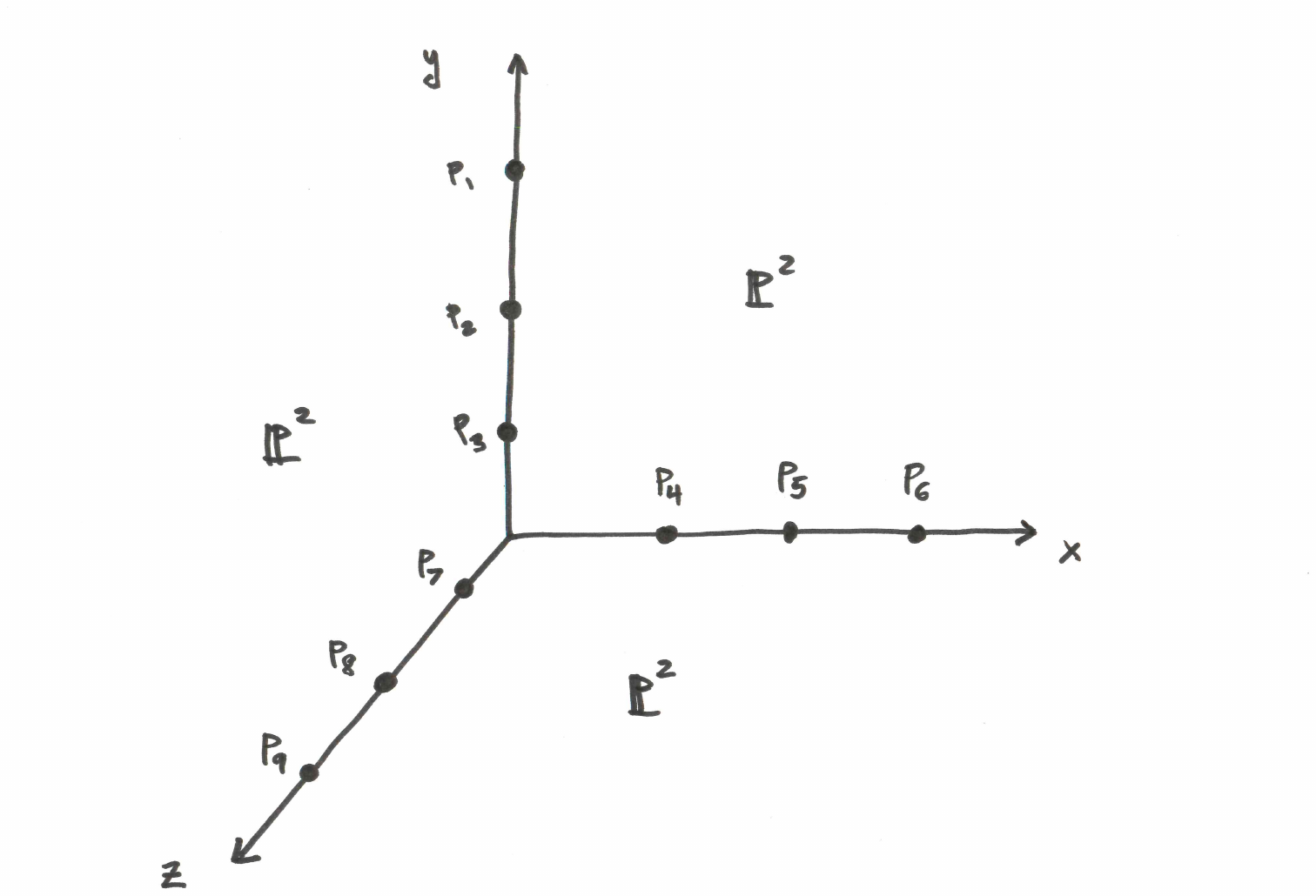}
	\caption{}\label{fThreePlanes}
\end{figure}

To desingularize $\YY$ we first blow up the $z=t=0$ plane in $\YY$. In the central fiber this
amounts to blowing up $P_1,\dots,P_6$ in this plane. Next we blow up the strict transform
of the $x=t=0$ plane which amounts to blowing up $P_7,P_8,P_9$ in that plane. The third plane
is left unchanged.
We thus obtain
a strictly semistable degeneration 
\[
	\pi \colon \XX \to \A^1
\]
whose central fiber $X = X_1 \cup X_2 \cup X_3$ is skeched in Figure \ref{fThreeBlownUpPlanes}, where we
have denoted the exceptional
divisor over $P_i$ by $E_i$.

\begin{figure}[h]
	\centering
	\includegraphics[scale=0.45]{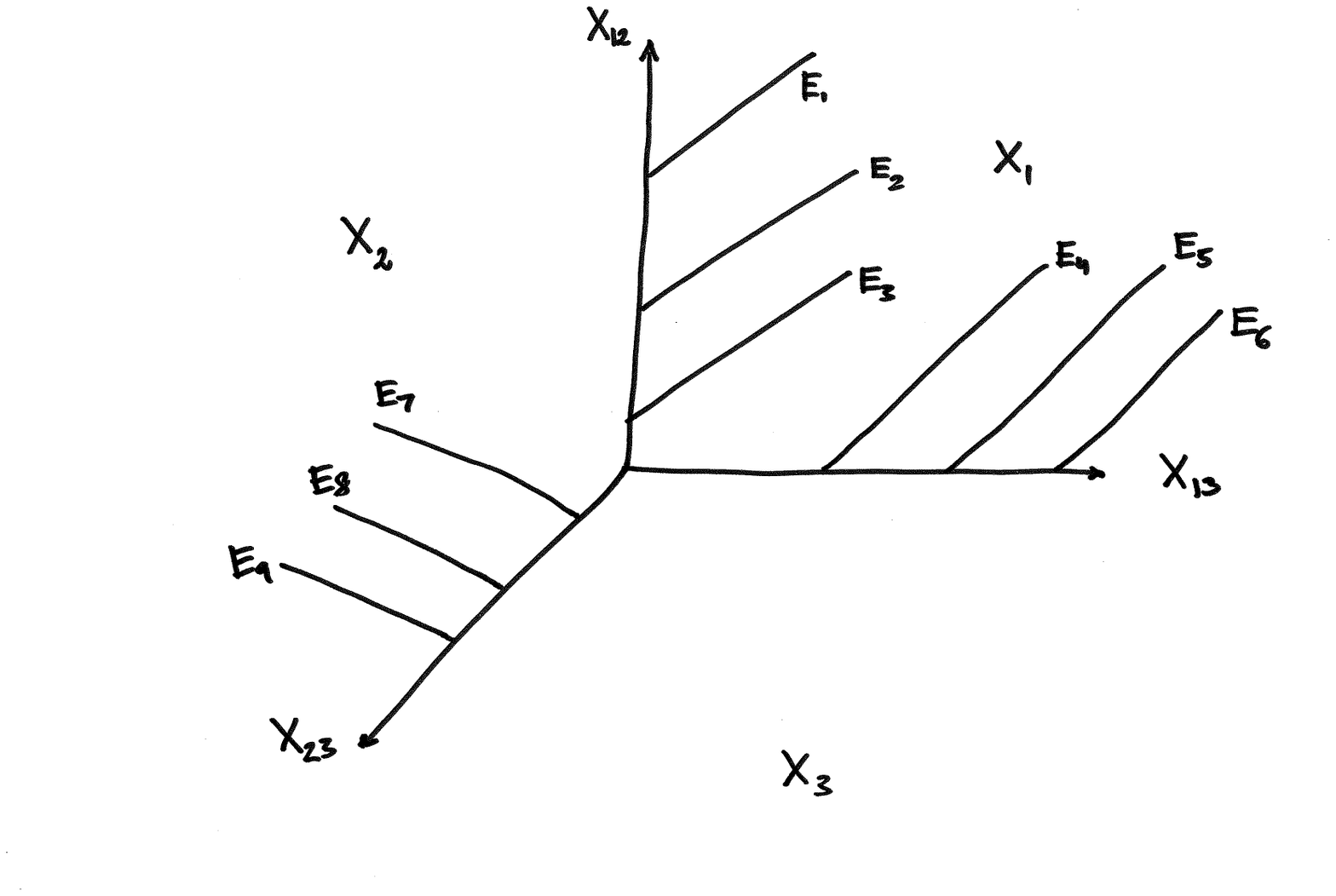}
	\caption{}\label{fThreeBlownUpPlanes}
\end{figure}

\begin{proposition}\label{pCubicSurfaceDegeneration}
 In this situation
\[
	\Chow_{\mathrm{prelog}}^1(X) = \Chow_{\mathrm{prelog,sat}}^1(X) = \Z^7.
\]
which nicely coincides with $\Chow^1(S)$.
\end{proposition}

\begin{proof}
For the convenience of the reader we give a slow walk through the necessary computations.
A Macaulay2 script doing the same work is available at \cite{BBG-M2}.

We calculate the codimension $1$ part of $\Chowprelog(X)$. Here we have
\[
	\bigoplus_{i=1}^3 \Chow_1(X_i) 
	= \langle H_1,E_1,\dots,E_6,H_2,E_7,\dots,E_9,H_3 \rangle
	= \Z^{12}
\]
The intersections of two components is always a $\P^1$ whose Chow group
we can identify with $\Z$ via the degree map. 
\[
	\bigoplus_{1 \le i < j \le 3} \Chow_1(X_{ij}) = \Z^3.
\]
Finally $X_{1,2,3}$ is a point with $\Chow^0(X_{1,2,3}) = \Z$. The diagram \ref{pBothmer} is
therefore 

\xycenter{
  & 0 \ar[d] & 0 \ar[d] & \\
  & R(X) \ar[d] \ar[r] & \mathrm{CH}^1_{\mathrm{prelog}}(X) \ar[d] \ar[r] & 0 \\
 \Z^3\ar[d]^-{\rho'} \ar[r]^\delta
& \Z^{12} \ar[r]^-{\nu_*} \ar[d]^-\rho
&\Chow_1(X) \ar[r]
& 0 \\
\Z \ar[r]^-{\delta'}
&\Z^3
 }

Now notice that the image of $X_{1,2}$ in $X_1$ has the class
\[
	[\ioUpper{1,2}{1}(X_{1,2})] = H_1-E_1-E_2-E_3,
\]
while its image in $X_2$ has the class
\[
	[\ioUpper{1,2}{2}(X_{1,2})] = H_2.
\]
Similarly
\begin{align*}
	[\ioUpper{1,3}{1}(X_{1,3})] &= H_1-E_4-E_5-E_6, \\
	[\ioUpper{1,3}{3}(X_{1,3})] &= H_3, \\
	[\ioUpper{2,3}{2}(X_{2,3})] &= H_2-E_7-E_8-E_9, \\
	[\ioUpper{2,3}{3}(X_{2,3})] &= H_3.
\end{align*}	
With this we obtain
\[
	\delta = {\left({\begin{array}{cccccccccccc}
      1&{-1}&{-1}&{-1}&0&0&0&{-1}&0&0&0&0\\
      1&0&0&0&{-1}&{-1}&{-1}&0&0&0&0&{-1}\\
      0&0&0&0&0&0&0&1&{-1}&{-1}&{-1}&{-1}\\
      \end{array}}\right)}.
\]
Here for example the first line says that $X_{1,2}$ is mapped to $(H_1-E_1-E_2-E_3)-H_2$.
We also get
\[
	\rho =  {\left({\begin{array}{cccccccccccc}
      1&1&1&1&0&0&0&{-1}&0&0&0&0\\
      1&0&0&0&1&1&1&0&0&0&0&{-1}\\
      0&0&0&0&0&0&0&1&1&1&1&{-1}\\
      \end{array}}\right)}^t.
\]
Here the first column says that $H_1$ intersects $X_{1,2}$ and $X_{1,3}$ but not $X_{2,3}$.
Notice also that the sign convention says that the restriction of a divisor on $X_1$ is positive on 
both $X_{1,2}$ and $X_{1,3}$ because $1$ is the smaller index in both cases. A divisor 
in $X_2$ is restricted with negative sign to $X_{1,2}$ but with positive sign to $X_{2,3}$.  A divisor
on $X_3$ is restricted with negative sign to both $X_{1,3}$ an $X_{2,3}$. 
Finally we get
\[
	\delta' = \begin{pmatrix}
      {-1}&1&{-1}\end{pmatrix} 
      \quad \text{and} \quad
      	\rho =  \begin{pmatrix}
      1&{-1}&1\end{pmatrix}^t.
\]
A good check that we got everything right is that indeed $\delta\rho = \rho' \delta'$ (this is
the Friedman condition).

Now one can check that $\delta$ is injective and that the image of $\delta$ 
is saturated in $\Z^{12}$. This can be done for example by checking that the gcd of the
$3 \times 3$ minors is $1$. Similarly one can check that $\rho$ is surjective.

This shows that the diagram reduces to
\xycenter{
 & 0 \ar[d] & 0 \ar[d] & \\
 & \Z^9 \ar[d]^{\varphi} \ar[r] & \mathrm{CH}^1_{\mathrm{prelog}}(X) \ar[d] \ar[r] & 0 \\
 \Z^3\ar[d]^-{\rho'} \ar[r]^\delta
& \Z^{12} \ar[r]^-{\nu_*} \ar[d]^-\rho
&\Z^9\ar[r]
& 0 \\
\Z \ar[r]^-{\delta'}
&\Z^3
 }

We can compute a representative matrix for $\varphi$ by calculating 
generators for the kernel of $\rho$. Since the image of $\delta$ is saturated
in $\Z^{12}$ a representative matrix for $\nu_*$ is obtained by 
\[
	\ker(\delta^t)^t.
\]
We can thus calculate $\varphi \nu_*$. One can check that this matrix has rank $7$ and
the image is saturated in $\Z^9$. Therefore
\[
	\Chow_{\mathrm{prelog}}^1(X) = \Chow_{\mathrm{prelog,sat}}^1(X) = \Z^7.
\]
\end{proof}

\begin{figure}[h]
	\centering
	\includegraphics[scale=0.45]{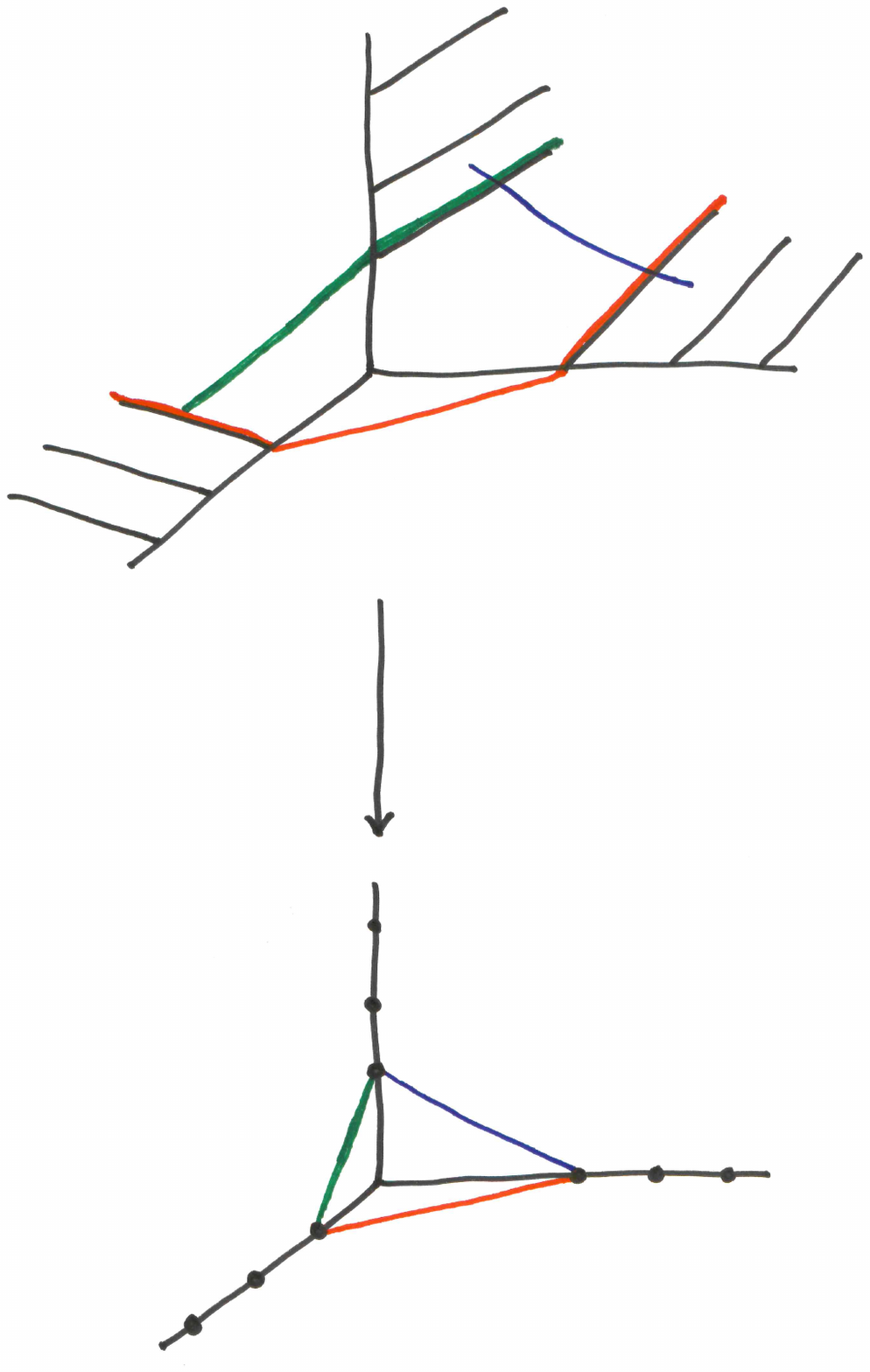}
	\caption{The 27 prelog lines of Proposition \ref{PropositionPrelogLines}.}\label{fPrelogLines}
\end{figure}

We now recall the classical log-geometric count of the 27 lines on a cubic surface. Assume that
we have a degeneration
\xycenter{
	\LL \ar[r] \ar[dr] & \XX \ar[d] \ar[r]& \YY \ar[dl] \\
	 & \A^1
}
with the generic fiber of $\LL$ a line in a cubic surface. The special fiber $L \subset Y$ is then a line
in one of the coordinate planes. Its preimage $L \subset X$ is a prelog cycle. Because of the prelog 
condition $L$ cannot intersect the coordinate lines outside of the singular points. Therefore we have on 
each of the 3 planes $3 \times 3$ possibilities to connect $2$ singularities on the $2$ adjacent coordinate lines. So in total we have $3 \cdot 3 \cdot 3 = 27$ possible ``lines'' on $Y$. Using log geometry one can 
then also prove that each of these lines indeed comes from a line on $S$ and that each line in $Y$ occurs only once as a limit. 

Here we exhibit the $27$ prelog cycles on $X$ that map to the above $27$ lines on $Y$:

\begin{proposition}\label{PropositionPrelogLines}
The following $27$ cycles on $X$ are prelog cycles (see Figure \ref{fPrelogLines}):
\begin{align*}
	(H_1-E_i-E_j,0,0) &\quad i\in\{1,2,3\}, j\in\{4,5,6\} \\
	(E_i,H_2-E_j,0) &\quad i\in\{1,2,3\}, j\in\{7,8,9\} \\
	(E_i,E_j,H_3) &\quad i\in\{4,5,6\}, j\in\{7,8,9\}
\end{align*}
These are the preimages in $X$ of lines connecting two singularities of $\YY$ in $Y$.
One can choose $7$ such cycles that generate $\Chow^1_{\mathrm prelog}(X)$.
\end{proposition}

\begin{proof}
In the first case $H_1-E_i-E_j$ intersects neither $H_1-E_1-E_2-E_3$, the images of $X_{1,2}$
in $X_1$, nor $H_1-E_4-E_5-E_6$, the image of $X_{1,3}$ in $X_1$. Therefore there need not be cycles on $X_2$ or $X_3$ matching with it on $X_{1,2}$ or $X_{1,3}$.

In the second case $E_j$ intersects the image of $X_{1,2}$ but not the image of $X_{1,3}$ on $X_1$. At the same time $H_2-E_j$ intersects $X_{1,2}$ but not $X_{2,3}$ on $X_2$. Therefore there need not be a cycle on $X_3$ matching with them on the intersection. A similar reasoning proves the third case. 

For a given set of $7$ lines one can check the assertion by calculating their images in $\Chow^1(X) = \Z^9$ and check whether they form a saturated sublattice of rank $7$. A computer program can easily find
a set of $7$ cycles that works. One such is exhibited in \cite{BBG-M2}.
\end{proof}

\section{Degenerations of self-products of elliptic curves and their prelog Chow rings}\label{sEllipticCurves}

We continue with the example of the self-product of an elliptic curve that will serve as an illustration of the degeneration method used to prove stable irrationality. 
Certainly a smooth elliptic curve $E$ is not stably rational and does not have a decomposition of the diagonal. To see the latter, suppose you could write, for rational or even just homological equivalence,
\[
\Delta_E = E\times p + \sum_i a_i (q_i \times E)
\]
with $p, q_i$ points on $E$ and $a_i\in \Z$. We restrict to homological equivalence for simplicity. Intersecting both sides of the equation with $E\times q$ for another point $q\neq p$, we find that $q = \sum a_i q_i$ and we may thus assume that 
\[
\Delta_E = E\times p + q \times E.
\]
Then, intersecting this equation with the cycle $T = \{ (x, x + p_0) \mid x\in E\}$, where $+p_0$ is translation on $E$ by a point $p_0$ different from zero, gives a contradiction: the left hand side results in a class of degree $0$ and the right hand side in a class of degree $2$.

\begin{figure}[h]
	\centering
	\includegraphics[scale=0.45]{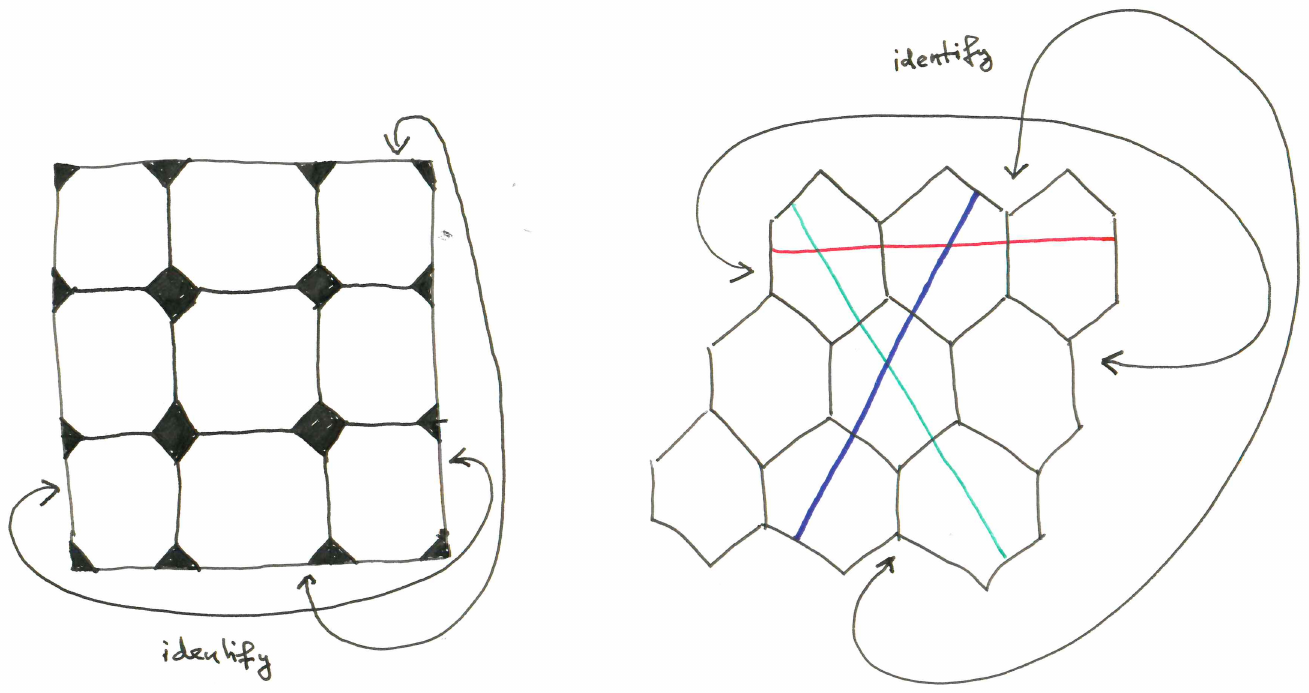}
	\caption{}\label{EllipticCurveDegeneration}
\end{figure}

Whereas this is well known and easy, it is nevertheless reassuring that we can also deduce it from Theorem \ref{tPrelogDecompositionInherited}, considering a degeneration $\pi_{\VV}\colon \VV \to C$ of plane cubic curves into a triangle of lines $V$ (we keep the notation of the preceding Section). Here $\pi_{\VV}$ is strictly semistable, but if we form $\VV\times_{C}\VV \to C$, the total space has singularities at points where four components of the central fibre intersect (which also then fails to be simple normal crossing).

This gives rise to nine nodes of the total space. Blowing these up, we get the picture schematically depicted on the left of Figure \ref{EllipticCurveDegeneration}: each node gets replaced by a $\P^1\times\P^1$ that is a component of the central fibre. However, as components of the new scheme-theoretic central fibre, these $\P^1\times\P^1$'s will have multiplicity $2$, and the resulting family is not strictly semistable. Therefore we contract the lines of one ruling of each of these quadrics to arrive at the picture on the right of Figure \ref{EllipticCurveDegeneration} (we contract in the ``North-West South-East direction" in the Figure). Here each hexagon is a $\P^2$ with three points blown up, and the intersection of two such ``hexagons" is a $(-1)$-curve in each of the surfaces. In this way we obtain a strictly semistable modification $\pi\colon \XX \to C$ of the product family. Here we do not follow the resolution scheme suggested in \cite[Prop 2.1]{Har01} as our method leads to a more symmetric central fibre. The green line in Figure \ref{EllipticCurveDegeneration} indicates the specialization of the diagonal, and the red line that of ``elliptic curve times a point". The blue line is the specialization of ``point times elliptic curve".

On each one of the hexagons, $\P^2$'s blown up in three points, there is thus a basis of the Picard group consisting of the pullback $H$ of the hyperplane class, and the three exceptional divisors $E_1, E_2, E_3$. One can identify the boundary components of each hexagon, proceeding in clockwise direction, with $E_1, H-E_1-E_2, E_2, H-E_2-E_3, E_3, H-E_1-E_3$.

\begin{proposition}\label{pPrelogSatElliptic}
The following hold:
\begin{enumerate}
\item
The classes of the red, green and blue lines in Figure  \ref{EllipticCurveDegeneration} generate $\mathrm{CH}_{\mathrm{prelog}}^1 (X)$ modulo torsion.
\item
The classes of the red, green and blue lines in Figure  \ref{EllipticCurveDegeneration} together with the half of their sum generate $\Chow_{\mathrm{prelog,sat}}^1(X)$.
\end{enumerate}
\end{proposition}

\begin{proof}
Part a) is a computation of the same type as in the proof of Proposition \ref{pCubicSurfaceDegeneration} and can be found in \cite{BBG-M2}. We observe that in this case the diagram of Proposition \ref{pBothmer} is of the form

\xycenter{
 & 0 \ar[d] & 0 \ar[d] & \\
 & \Z^{11} \ar[d]^{\varphi} \ar[r] & \mathrm{CH}^1_{\mathrm{prelog}}(X) \ar[d] \ar[r] & 0 \\
 \Z^{27}\ar[d]^-{\rho'} \ar[r]^\delta
& \Z^{36} \ar[r]^-{\nu_*} \ar[d]^-\rho
&\Z^{11}\oplus \Z/3\ar[r]
& 0 \\
\Z^{18} \ar[r]^-{\delta'}
&\Z^{27}
 }
Let $\bar{\nu}_*$ be the composition of $\nu_*$ with the projection to $\Z^{11}$. The composition $\varphi\circ \bar{\nu}_*$ is an $11\times 11$ matrix representing $\mathrm{CH}_{\mathrm{prelog}}^1 (X)$ modulo torsion. A computation done in  \cite{BBG-M2} shows that the rank of this matrix is $3$ in every characteristic except $\mathrm{char}=2$ where it is $2$. The sum of the red, green and blue lines is divisible by $2$, which gives b).
\end{proof}

\begin{corollary}\label{cDecEll}
There is no prelog decomposition of the diagonal for $X$ relative to the family $\pi\colon\XX\to C$ satisfying the additional stronger condition of Remark \ref{rStrengthening}. 
\end{corollary}

\begin{proof}
The class $\zeta$ of Definition \ref{dPrelogDecompositionDiagonal} is the difference between the green and red lines, regardless of the cover $C'\to C$ we may have to pass to.
If $(A_i)$ is a prelog cycle on $X$ satisfying (3) of Definition \ref{dPrelogDecompositionDiagonal}, then applying $\varrho\circ \mathrm{pr}_1$ to $\sum_i \nu_* A_i$ is a union of points in $V$. We want to show that the cycle class of $\sum_i \nu_* A_i$ must be a multiple of the cycle class of the blue line. This will give a contradiction since the classes of the blue, green and red lines are independent over $\Z$. For this it is sufficient to consider the case where $\sum_i \nu_* A_i$ maps to one point $p$ in $V$. It is clear that the support of $\sum_i \nu_* A_i$ is contained in the fibre $(\varrho\circ \mathrm{pr}_1)^{-1}(p)$. We distinguish two cases: (1) $p$ is in the smooth part of the triangle of lines $V$; (2) $p$ is a vertex of that triangle.  
In case (1), we obviously have to take all components of the fibre with equal multiplicities to get a prelog cycle (which is depicted in Figure \ref{EllipticCurveDegeneration}, blue line). In case (2), consider Figure \ref{Picture42}.

\begin{figure}[h]
	\centering
\begin{tikzpicture}[smooth, scale=1.2]
%
\node at (1.1,-0.1) {$\scriptstyle{\color{blue} A_6}$};
\node at (2.1,1.9) {$\scriptstyle{\color{blue} A_4}$};
\node at (2.9,3.75) {$\scriptstyle{\color{blue} A_2}$};
\node at (1.4,0.5) {$\scriptstyle{\color{blue} A_5}$};
\node at (2.4,2.5) {$\scriptstyle{\color{blue} A_3}$};
\node at (3.4,4.8) {$\scriptstyle{\color{blue} A_1}$};
\node at (4.6,5.2) {$\scriptstyle{\color{blue} B_6}$};
\node at (2.5,1.1) {$\scriptstyle{\color{blue} B_4}$};
\node at (3.7,3.3) {$\scriptstyle{\color{blue} B_2}$};
\node at (2.2,0.5) {$\scriptstyle{\color{blue} B_5}$};
\node at (3.2,2.5) {$\scriptstyle{\color{blue} B_3}$};
\node at (4.2,4.8) {$\scriptstyle{\color{blue} B_1}$};
\foreach \x in {2,4,6}
\draw[thick, red] (\x,4.5) to (\x+1.6,4.5);
\foreach \x in {2,3,4}
\draw[thick, lime] (\x+0.4,9.3-2*\x) to (\x+1.2,7.7-2*\x);
\foreach \x in {0,2,4}
\foreach \y in {0}
   \draw (\x,\y) to (\x+0.8,\y-0.6) to (\x+1.6,\y) to (\x+1.6,\y+1) to (\x+0.8,\y+1.6) to (\x,\y+1) to (\x,\y);
\foreach \x in {1,3,5}
\foreach \y in {2}
   \draw (\x,\y) to (\x+0.8,\y-0.6) to (\x+1.6,\y) to (\x+1.6,\y+1) to (\x+0.8,\y+1.6) to (\x,\y+1) to (\x,\y);
\foreach \x in {2,4,6}
\foreach \y in {4}
   \draw (\x,\y) to (\x+0.8,\y-0.6) to (\x+1.6,\y) to (\x+1.6,\y+1) to (\x+0.8,\y+1.6) to (\x,\y+1) to (\x,\y);
\foreach \x in {0,1,2}
   \draw [snake=snake, segment amplitude=.4mm, segment length=0.4mm] (\x,2*\x) -- (\x+0.8,2*\x-0.6);
\foreach \x in {0,1,2}
   \draw [snake=snake, segment amplitude=.4mm, segment length=0.4mm] (\x+0.8,2*\x+1.6) -- (\x+1.6,2*\x+1);
\foreach \x in {2,3,4}
   \draw [snake=snake, segment amplitude=.4mm, segment length=0.4mm] (\x,2*\x-4) -- (\x+0.8,2*\x-4.6);
\foreach \x in {2,3,4}
   \draw [snake=snake, segment amplitude=.4mm, segment length=0.4mm] (\x+0.8,2*\x-2.4) -- (\x+1.6,2*\x-3);
\foreach \x in {0,1,2}
   \draw[very thick, blue] (\x+0.8,2*\x-0.6) to (\x+1.6,2*\x) to (\x+1.6,2*\x+1);
\foreach \x in {2,3,4}
   \draw[very thick, blue] (\x+0.8,2*\x-2.4) to (\x,2*\x-3) to (\x,2*\x-4);
\end{tikzpicture}
	\caption{}\label{Picture42}
\end{figure}

The prelog cycle can be written as
\[
X=\sum_{i=1}^6 \alpha_i A_i + \sum_{j=1}^6 \beta_j B_j .
\]
The prelog condition at the shaded edges gives:
\begin{gather*}
\alpha_2=\alpha_3, \: \alpha_4=\alpha_5, \: \alpha_6=\alpha_1, \\
\beta_1=\beta_2, \: \beta_3 =\beta_4, \: \beta_5=\beta_6 .
\end{gather*}
Moreover, $\nu_* X$ is in the prelog Chow ring. Let $r, g$ be the intersection numbers of $\nu_* X$ with the element given by the red and green cycles, respectively. We have 
\begin{gather*}
\alpha_1 + \beta_1 = r 
\end{gather*}
and by Proposition \ref{pIntersectionPairingPrelog},  moving the red cycles to the second and third rows of hexagons (i.e., replacing them by rationally equivalent ones) we also get
\[
\alpha_3 + \beta_3 = \alpha_5+ \beta_5=  r .
\]
Similarly, considering the green cycles we find
\[
\alpha_2 +\beta_2 =\alpha_4+\beta_4 = \alpha_6+\beta_6 =g. 
\]
The equations imply that $r$ is equal to $g$ and all $\alpha_i$'s are equal to each other, and the same holds for the $\beta_j$'s.
\end{proof}

\appendix

\section{Interaction with the Gross-Siebert programme}\label{sTropicalDegenerations}

In this section we reformulate the two main examples of this paper in the language of the Gross-Siebert programme. This concerns the degeneration of a cubic surface and the degeneration of the self-product of an elliptic curve. Our goal here is to illustrate how the prelog Chow rings fit with the constructions of the Gross-Siebert programme. We do not give proofs as the results are already proven in the preceding sections. We introduce the concepts and construction of the Gross-Siebert programme that we need and refer the interested reader to \cite{CPS10,Gra20,GHS20} for more details.

While we do not strive for maximal generality, it will be clear how the examples naturally generalize to degenerations of hypersurfaces in smooth projective toric varieties and to degenerations of abelian varieties. The families are described by combinatorial data and we will see how 1-cycles in the general fiber correspond to tropical curves in the dual intersection complex.

For cubic surfaces, we start with the polytope of $\P^3$ polarized by $\mathcal{O}(3)$ as given by the convex hull in $\R^3$ of $(-1,-1,-1)$, $(2,-1,-1)$, $(-1,2,-1)$ and $(-1,-1,2)$. We take its integral regular subdivision obtained by adding the convex hull of each vertex with the edge joining $(-1,-1,-1)$ with $(0,0,0)$. This determines a degeneration of $\P^3$ into the union of three copies of $\P^3$ each polarized by $\mathcal{O}(1)$. Moreover, a generic section of $\mathcal{O}_{\P^3}(3)$ degenerates into the union of three hyperplanes.
This degeneration is however not strictly semistable, there are 9 \emph{focus-focus singularities}, 3 each on the intersection of the hyperplanes, cf.\ Figure \ref{fintercomplexcubic}. These singularities correspond to the locus where the family is not log smooth (and hence also not strictly semistable).

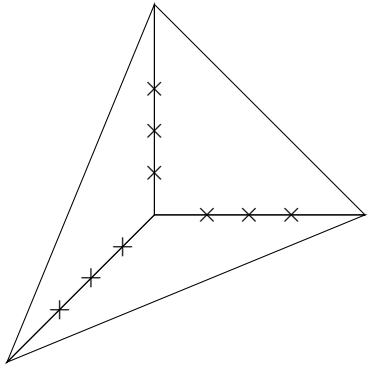
\begin{figure}[h]
\begin{tikzpicture}[smooth, scale=1.4]
\draw (0,0) to (0,2) to (-1.4,-1.4) to (0,0);
\draw (0,0) to (0,2) to (2,0) to (0,0);
\draw (0,0) to (-1.4,-1.4) to (2,0) to (0,0);
\node at (-0.3,-0.3) {$+$};
\node at (-0.6,-0.6) {$+$};
\node at (-0.9,-0.9) {$+$};
\node at (0,0.4) {$\times$};
\node at (0,0.8) {$\times$};
\node at (0,1.2) {$\times$};
\node at (0.5,0) {$\times$};
\node at (0.9,0) {$\times$};
\node at (1.3,0) {$\times$};
\end{tikzpicture}
\caption{The intersection complex for the degeneration of a cubic surface. There are 9 focus-focus singularities along the inner edges.}
\label{fintercomplexcubic}
\end{figure}

One could now proceed by resolving the singularities as in Section \ref{sExampleCubic}. For our purposes, we however already have enough information to read off the $(-1)$-curves from the dual picture, namely the dual intersection complex of Figure \ref{fdualintercomplexcubic}. Indeed, the $(-1)$-curves correspond to the lowest degree tropical curves of \cite{CPS10,Gra20}. These are formed by choosing two of the compact edges and on each edge a focus-focus singularity. From each of the singularities, a ray is emitted along the compact edge into the direction opposite to the third compact edge. When they meet, they combine according to the balancing condition to a ray that is identified with the corresponding ray of the fan. Clearly, there are 27 of them.

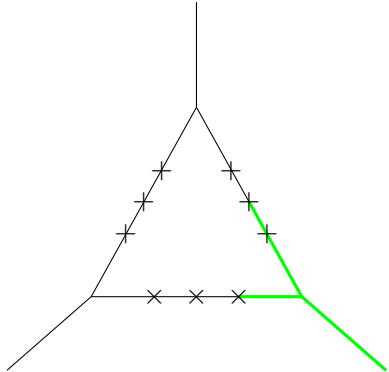
\begin{figure}[h]
\begin{tikzpicture}[smooth, scale=1.4]
\draw (-1,-1) to (1,-1) to (0,0.8) to (-1,-1);
\draw (-1.8,-1.7) to (-1,-1);
\draw (1,-1) to (1.8,-1.7);
\draw (0,0.8) to (0,1.8);
\draw[very thick, green] (0.5,-0.1) to (1,-1) to (0.4,-1);
\draw[very thick, green] (1,-1) to (1.8,-1.7);
\node at (-0.67,-0.4) {$+$};
\node at (-0.5,-0.1) {$+$};
\node at (-0.33,0.2) {$+$};
\node at (0.67,-0.4) {$+$};
\node at (0.5,-0.1) {$+$};
\node at (0.33,0.2) {$+$};
\node at (-0.4,-1) {$\times$};
\node at (0,-1) {$\times$};
\node at (0.4,-1) {$\times$};
\end{tikzpicture}
\caption{A sketch of the dual intersection complex for the degeneration of a cubic surface. Locally around the 3 vertices, the affine structure is given by the fan of $\P^2$. Each of the edges shared by two of the $\P^2$ has 3 focus-focus singularities on it. One of the 27 $(-1)$-curves is drawn in green. For the integral affine manifold with singularities (dual intersection complex) that correctly captures the cuts induced by the focus-focus singularities, see \cite[Figure 1.5(3a)]{Gra20}.}
\label{fdualintercomplexcubic}
\end{figure}

We turn to degenerations of abelian varieties, in particular the degeneration of $E\times E$ from Section \ref{sEllipticCurves}. They fit into the framework of degenerations of abelian varieties as studied by Mumford \cite{Mum72} and Alexeev \cite{Al02}. We follow the general framework by Gross-Hacking-Siebert \cite[Section 6]{GHS20}. While we do not give a full treatment, it will be clear how strictly semistable degenerations as in \cite[Section 6]{GHS20} can be used to study Chow groups of abelian varieties. Moreover, the examples below illustrate the concepts and results of this paper. The degenerations are encoded by combinatorial information on integral affine tori.

\begin{definition}
A differentiable manifold $B$ is an \emph{integral affine manifold} if its transition functions are in ${\rm Aff}(\Z^n):=\Z^n\rtimes{\rm GL}(\Z^n)$. The set $B(\frac{1}{d}\Z)\subseteq B$ of \emph{$1/d$-integral points} is given chartwise locally by the preimage of $\frac{1}{d}\Z^n\subseteq\R^n$. $B$ comes with a sheaf of integral tangent vectors $\Lambda$, resp.\ cotangent vectors $\check{\Lambda}$.

\end{definition}

Let $n\in\N$, let $M\cong\Z^n$ be a lattice and write $M_{\R} := M\otimes_\Z\R$. Let $\Gamma \subseteq M$ be a rank $n$ sublattice. Then the real $n$-torus $B:=M_{\R}/\Gamma$ inherits an integral affine structure from $M_{\R}$. Moreover, a $\Gamma$-periodic integral polyhedral decomposition $\overline{\PP}$ of $M_\R$ induces an integral polyhedral decomposition $\PP$ of $B$. We call the elements of $\PP$ the \emph{cells} of $B$. For a cell $\tau$, we may consider the restriction $\Lambda_\tau$, resp.\ $\check{\Lambda}_\tau$, of $\Lambda$, resp.\ $\check{\Lambda}$, to $\tau$.

This datum is enough to determine the central fiber $\XX_0$. Each cell $\tau$ is an integral polyhedron and hence determines a toric variety $\P_\tau$ and $\XX_0$ is given by gluing the $\P_\tau$ as prescribed by $\PP$.
Next we describe the deformation of $\XX_0$. We are interested in 1-parameter deformations and restrict to that generality here. In the notation of \cite{GHS20}, this means setting $Q=\N$.

\begin{definition}
A \emph{piecewise affine function} on an open set $U\subseteq B$ is a continuous map $U \to \R$ which restricts to an integral affine function on each maximal cell of $\PP$. The sheaf of piecewise affine functions is denoted by $\mathcal{PA}(B)$. 
The sheaf $\mathcal{MPA}(B)$ of \emph{multivalued piecewise affine (MPA-) functions} on $B$ is the quotient of $\mathcal{PA}(B)$ by the sheaf of affine functions on $B$. 

\end{definition}

Let $\varphi$ be a MPA-function on $B$, i.e.\ a global section of $\mathcal{MPA}(B)$. Let $\rho$ be a codimension 1 cell separating two maximal cells $\sigma$, $\sigma'$ and let $P\in\rho$ be an interior point. Choose a primitive vector $v\in\Lambda_P$ that complements an integral basis of $\Lambda_{\rho,P}$ to an integral basis of $\Lambda_P$. Consider the integral span $L$ of $v$. We identify $L$ with $\Z$ by choosing the non-negative vectors to point into $\sigma'$. Choosing a representative, $\varphi$ induces a piecewise affine function $\varphi_P:L\otimes_\Z\R=\R\to\R$.
Denote by $n$, resp.\ $n'$, the slope of $\varphi_P$ on $\R_{\leq0}$, resp.\ $\R_{\geq0}$. The \emph{kink} of $\varphi$ along $\rho$ is defined to be $\kappa_\rho(\varphi):=n'-n\in\Z$.
As an MPA-function, $\varphi$ is determined by the set of kinks across codimension 1 cells. $\varphi$ is said to be \emph{convex} if all of its kinks are non-negative.

\begin{example}

Consider $\R^2$ with the standard integral structure and consider the polyhedral decomposition given by the diagonal $x=y$. Denote by $\sigma$, resp.\ $\sigma'$, the upper left, resp.\ bottom right, halfplane. Define the piecewise linear function $\varphi$ by setting $\varphi|_{\sigma}:=y$ and $\varphi|_{\sigma'}:=x$. Choose $P$ to be the origin. Then $L$ as above is for example the span of $(1,0)$ and $\varphi$ is seen to have kink 1.

\end{example}

Let $\tau$ be a cell of codimension 2 and let $\rho_1,\dots,\rho_k$ be the adjacent cells of codimension 1. We work in a chart around a vertex $P\in\tau$. For $1\leq i\leq k$, let $n_i\in\check{\Lambda}_P$ be a primitive cotangent vector normal to $\Lambda_{\rho_i}$. The signs of the $n_i$ are chosen by following a simple loop around the origin of $\Lambda_P\otimes_\Z\R / \Lambda_{\tau,P}\otimes_\Z\R $. The \emph{balancing condition} states that
\[
\sum_{i=1}^k \kappa_{\rho_i}(\varphi)\otimes n_i = 0.
\]
When the balancing condition is satisfied at each vertex $P$ of each codimension 2 cell $\tau$, $\varphi$ is locally single-valued.

While we give a global description of the family below, we can already write down local equations, which detect semistability. We denote the smoothing parameter by $t$. The family is defined order by order and is given over  ${\rm Spec}\, \C[t]/t^{k+1}$ for some $k$. To each cell $\tau$ we define a $\C[t]/t^{k+1}$-algebra $R_\tau$ and gluing morphisms according to \cite[Section 2.2]{GHS20}. The construction is explained by locally at $P$ taking the monoid $\widehat{M_P}$ which is the upper convex hall of the graph of $\varphi$, then the local rings are given as $\C[\widehat{M_P}]/t^{k+1}$, where $t$ is identified with $z^{0\oplus1}$, cf.\ below. We refer for details to \cite{GS11,GHS20} and only give an overview here. Note that there is no wall structure present, which greatly simplifies the construction. Locally, everything is toric and the gluing morphisms are given by identifying monomials in the coordinate rings of algebraic tori. The $\widehat{M_P}$ are the local monoids of the polytope $\Xi_\varphi$ below.

For a maximal cell $\sigma$, $R_\sigma:=\left(\C[t]/t^{k+1}\right)[\Lambda_\sigma]$ gives an algebraic torus over the base. At any point $P$ and for $v\oplus m\in\Lambda_P\oplus\Z$, we write $z^{v\oplus m}$ for the corresponding element in $\C[\Lambda_P\oplus\Z]=\C[\Lambda_P][t]$.
Let $P\in \rho$ be an interior point of a codimension 1 cell $\rho$ separating two cells $\sigma_+$ and $\sigma_-$. Let $v_+\in\Lambda_P$ be an integral vector pointing into $\sigma_+$ and completing an integral basis of $\Lambda_{\rho,P}$ to an integral basis of $\Lambda_{P}$. Denote by $v_-$ the opposite of $v_+$, pointing into $\sigma_-$ and write $Z_+=z^{v_+\oplus \,d\varphi_P(v_+)}$ and $Z_-=z^{v_-\oplus\,d\varphi_P(v_-)}$. Notice that $v_\pm\oplus \,d\varphi_P(v_\pm)\in\widehat{M_P}$ are elements of the boundary. Then
\[
R_\rho := \left(\C[t]/t^{k+1}\right)[\Lambda_\rho][Z_+,Z_-] \,{\big/}\left(Z_+Z_--t^{\kappa_\rho(\varphi)}\right),
\]
reflecting the fact that
\[
\left(v_+\oplus \,d\varphi_P(v_+)\right)+\left(v_-\oplus \,d\varphi_P(v_-)\right)=\kappa_\rho(\varphi)
\]
in $\widehat{M_P}$. Notice that compared to \cite[Equation (2.11)]{GHS20}, the wall-crossing functions $f_{\underline{\rho}}=1$ as we are only considering degenerations of abelian varieties.

We have just described the deformation family away from codimension $\geq 2$ cells. Looking at the shape of the equations, we conclude that in a neighborhood of an interior point of $\P_\rho$, the family is semistable. Notice that while the construction depends on choosing $v_+$, this dependency is removed when identifying $Z_+$ with the corresponding element in $R_{\sigma_+}$. Semistability at codimension $\geq 2$ cells is not guaranteed as the following discussion shows.

At points of cells in codimension 2, the equations depend both on the local structure of how the codimension 1 cells meet and the kinks. We describe two examples, which the reader will recognize from Section \ref{sEllipticCurves}.

\begin{example}\label{eexe}
Start with $M=\Z^2$ and let $\Gamma$ be generated by $(3,0)$ and $(0,3)$. Figure \ref{fpolydecomp2torus} gives a $\Gamma$-periodic polyhedral decomposition of the 2-torus. The central fiber is obtained by gluing copies of $\P^1\times\P^1$. For the deformation, we choose the kink to be 1 across all edges. Alternatively, we could have rescaled $\varphi$ by a factor of 2, which would correspond to choosing the kink to be 2 across all edges. The latter family is obtained by the base change $t\mapsto t^2$.

Locally around interior points of the edges, the family is given by $Z_+Z_-=t$. Locally around the vertex $(1,1)$, write $z_1=z^{(1,0)\oplus1}$, $z_2=z^{(-1,0)\oplus0}$, $z_3=z^{(0,1)\oplus1}$ and $z_4=z^{(0,-1)\oplus0}$. Then the local equation is $z_1z_2=z_3z_4=t$ and hence the family is not semistable.
\end{example}

\begin{figure}[h]
\begin{center}
\begin{tikzpicture}[smooth, scale=1.2]
\draw[step=1cm,gray,very thin] (-1.5,-1.5) grid (4.5,4.5);
\draw[thick] (0,0) to (3,0);
\draw[thick] (0,1) to (3,1);
\draw[thick] (0,2) to (3,2);
\draw[thick] (0,3) to (3,3);
\draw[thick] (0,0) to (0,3);
\draw[thick] (1,0) to (1,3);
\draw[thick] (2,0) to (2,3);
\draw[thick] (3,0) to (3,3);
\node at (0,0) {$\bullet$};
\node at (-0.3,-0.3) {$\scriptstyle{(0,0)}$};
\node at (3,0) {$\bullet$};
\node at (3.3,-0.3) {$\scriptstyle{(3,0)}$};
\node at (0,3) {$\bullet$};
\node at (-0.3,3.3) {$\scriptstyle{(0,3)}$};
\node at (3,3) {$\bullet$};
\node at (3.3,3.3) {$\scriptstyle{(3,3)}$};
\node at (0.5,0.5) {$\scriptstyle{1}$};
\node at (1.5,0.5) {$\scriptstyle{x}$};
\node at (0.5,1.5) {$\scriptstyle{y}$};
\node at (1.5,1.65) {$\scriptstyle{x+y}$};
\node at (1.5,1.35) {$\scriptstyle{-1}$};
\node at (2.5,0.5) {$\scriptstyle{2x-2}$};
\node at (0.5,2.5) {$\scriptstyle{2y-2}$};
\node at (2.5,1.65) {$\scriptstyle{2x+y}$};
\node at (2.5,1.35) {$\scriptstyle{-3}$};
\node at (1.5,2.65) {$\scriptstyle{x+2y}$};
\node at (1.5,2.35) {$\scriptstyle{-3}$};
\node at (2.5,2.65) {$\scriptstyle{2x+2y}$};
\node at (2.5,2.35) {$\scriptstyle{-5}$};
\end{tikzpicture}
\end{center}
\caption{A product polyhedral decomposition of the 2-torus with a representative of a convex MPA-function determined by choosing the kink to be 1 across all edges. The balancing condition is satisfied and the MPA-function is locally single-valued.}
\label{fpolydecomp2torus}
\end{figure}
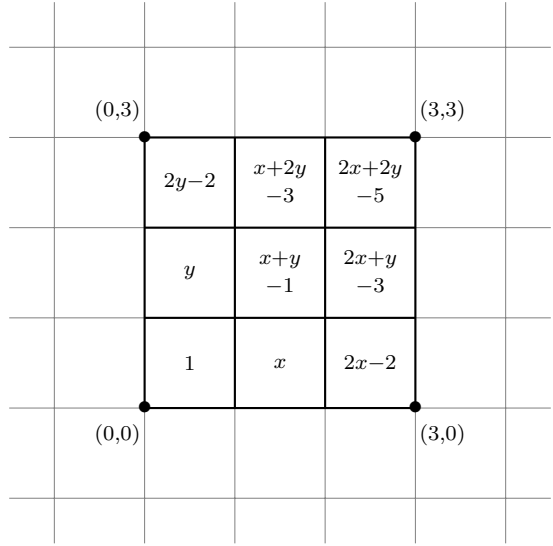


\begin{example}\label{eexestrict}
Start with $M=\Z^2$ and let $\Gamma$ be generated by $(6,3)$ and $(3,6)$. Figure \ref{fstrictlypolydecomp2torus} gives a $\Gamma$-periodic polyhedral decomposition of the 2-torus. The central fiber is given by gluing copies of degree 6 del Pezzo surfaces. For the deformation, we choose the kink to be 1 across all edges.

Locally around interior points of the edges, the family is given by $Z_+Z_-=t$. Locally around the vertex $(2,2)$, write $z_1=z^{(-1,0)\oplus0}$, $z_2=z^{(0,-1)\oplus0}$ and $z_3=z^{(1,1)\oplus1}$. Then the local equation is $z_1z_2z_3=t$ and similarly for other vertices. We conclude that the family is semistable.
\end{example}

\begin{figure}[h]
\begin{center}
\begin{tikzpicture}[smooth, scale=0.9]
\draw[step=1cm,gray,very thin] (-2.5,-2.5) grid (9.5,9.5);
\draw[thick,red] (-1,-1) to (0,-1);
\draw[thick] (0,-1) to (1,0) to (2,0) to (3,1) to (4,1);
\draw[thick,green] (4,1) to (5,2);
\draw[thick,blue] (5,2) to (5,3);
\draw[thick,red] (5,2) to (6,2);
\draw[thick,blue] (-1,-1) to (-1,0);
\draw[thick] (-1,0) to (0,1) to (1,1) to (2,2) to (3,2) to (4,3) to (5,3);
\draw[thick] (0,1) to (0,2) to (1,3) to (2,3) to (3,4) to (4,4) to (5,5) to (6,5) to (6,4) to (5,3);
\draw[thick] (1,3) to (1,4);
\draw[thick,green] (1,4) to (2,5);
\draw[thick,red] (2,5) to (3,5);
\draw[thick,blue] (2,5) to (2,6);
\draw[thick] (3,5) to (4,6) to (5,6) to (6,7) to (7,7) to (7,6) to (6,5);
\draw[thick] (1,0) to (1,1);
\draw[thick] (3,1) to (3,2);
\draw[thick] (2,2) to (2,3);
\draw[thick] (4,3) to (4,4);
\draw[thick] (3,4) to (3,5);
\draw[thick] (5,5) to (5,6);
\draw[thick,green] (7,7) to (8,8);
\draw[thick,red] (8,8) to (9,8);
\draw[thick,blue] (8,8) to (8,9);
\draw[thick,green] (-2,-2) to (-1,-1);
\node at (-1,-1) {$\bullet$};
\node at (-0.6,-0.7) {$\scriptstyle{(0,0)}$};
\node at (5,2) {$\bullet$};
\node at (5.4,2.3) {$\scriptstyle{(6,3)}$};
\node at (2,5) {$\bullet$};
\node at (2.4,4.7) {$\scriptstyle{(3,6)}$};
\node at (8,8) {$\bullet$};
\node at (8.4,8.3) {$\scriptstyle{(9,9)}$};
\node at (0.1,0.2) {$\scriptstyle{2}$};
\node at (2.1,1.2) {$\scriptstyle{x}$};
\node at (1.1,2.2) {$\scriptstyle{y}$};
\node at (4.1,2.2) {$\scriptstyle{2x-4}$};
\node at (2.1,4.2) {$\scriptstyle{2y-4}$};
\node at (3.1,3.2) {$\scriptstyle{x+y-3}$};
\node at (5.1,4.2) {$\scriptstyle{2x+y-8}$};
\node at (4.1,5.2) {$\scriptstyle{x+2y-8}$};
\node at (6.2,6.2) {$\scriptstyle{2x+2y-14}$};
\draw [<->,red] (0.7,3.5) to [out=160,in=40] (7.3,6.8);
\draw [<->,red] (-0.9,-1.3) to [out=210,in=80] (2.2,5.2);
\end{tikzpicture}
\end{center}
\caption{A polyhedral decomposition of the 2-torus with a representative of a convex MPA-function. The kink is 1 across horizontal or vertical edges, and 2 across diagonal edges, which determines the smoothing. The balancing condition is satisfied and the MPA-function is locally single-valued.}
\label{fstrictlypolydecomp2torus}
\end{figure}
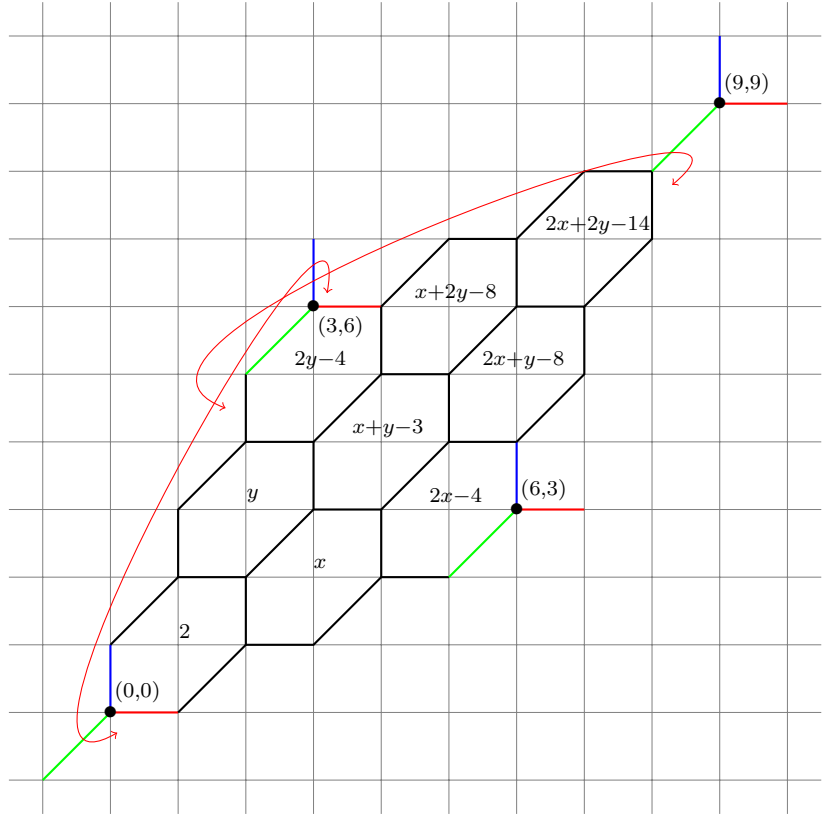

Next, we describe the global construction of \cite[Section 6]{GHS20} yielding a formal family $\widehat{\XX}\to{\rm Spf} \C[[t]]$ as the quotient of a non-finite type fan.
We will see how to detect the prelog 1-cycles in Example \ref{eexestrict} as well as how Example \ref{eexestrict} is a logarithmic modification of Example \ref{eexe}.

For $\gamma\in\Gamma$, the $\Gamma$-periodicity of $\varphi$ states that there is an affine linear function $\alpha_\gamma$ such that
\[
\varphi(x+\gamma)=\varphi(x)+\alpha_\gamma(x).
\]
Start with the polytope
\[
\Xi_\varphi := \left\{ (m, \varphi(m) + q) \, \big{|} \, q \in \R_{\geq0}  \right\} \subseteq M_\R \times \R,
\]
on which $\gamma\in\Gamma$ acts via
\[
(m,q) \mapsto (m+\gamma,q+\alpha_\gamma(m)).
\]
Next, consider the normal fan $\Sigma$ to $\Xi_\varphi$. $\Sigma$ is a fan in $N_\R\times\R$ with support contained in $N_\R\times\R_{\geq0}$, where $N={\rm Hom}(M,\Z)$. $\Sigma$ admits a natural map to the fan $\R_{\geq0}$ of $\A^1$ given by the projection to the last coordinate. $\Sigma$ is the cone over the faces $\check{\sigma}$ of the dual polyhedral decomposition $\check{\PP}$ of $N_\R$. Denoting by $X_\Sigma$ its associated toric variety, which is not of finite type, we thus have a flat morphism
\[
f: X_\Sigma \to {\rm Spec} \, \C[t].
\]
The generic fiber is the algebraic $n$-torus. Each ray of $\Sigma$, dual to a maximal cell $\sigma$ of $B$, corresponds to the component $\P_\sigma$ of the central fiber. More precisely, for $\sigma$ a maximal cell of $B$,
\[
\left\{  (m, \varphi(m) \, \big{|} \, m\in\sigma \right\}
\]
is the corresponding maximal face of $\Xi_\varphi$, and the primitive normal vector to this face is $\left(-d(\varphi|_\sigma)^t(1),1\right)$. From this description, we see that Example \ref{eexe} is the product family of a degeneration of an elliptic curve into a cycle of three $\P^1$. It is also clear that scaling $\varphi$ by a factor $r$ corresponds to dilating $\check{\PP}$ by $r$.

\begin{figure}[h]
\begin{center}
\begin{tikzpicture}[smooth, scale=0.6]
\draw[step=1cm,gray,very thin] (-0.5,-0.5) grid (3.5,3.5);
\draw[thick] (-0.5,0) to (3.5,0);
\draw[thick] (-0.5,1) to (3.5,1);
\draw[thick] (-0.5,2) to (3.5,2);
\draw[thick] (-0.5,3) to (3.5,3);
\draw[thick] (0,-0.5) to (0,3.5);
\draw[thick] (1,-0.5) to (1,3.5);
\draw[thick] (2,-0.5) to (2,3.5);
\draw[thick] (3,-0.5) to (3,3.5);
\node at (0,0) {$\bullet$};
\node at (1,0) {$\bullet$};
\node at (2,0) {$\bullet$};
\node at (3,0) {$\bullet$};
\node at (0,1) {$\bullet$};
\node at (1,1) {$\bullet$};
\node at (2,1) {$\bullet$};
\node at (3,1) {$\bullet$};
\node at (0,2) {$\bullet$};
\node at (1,2) {$\bullet$};
\node at (2,2) {$\bullet$};
\node at (3,2) {$\bullet$};
\node at (0,3) {$\bullet$};
\node at (1,3) {$\bullet$};
\node at (2,3) {$\bullet$};
\node at (3,3) {$\bullet$};
\draw [<->,red] (-0.7,3) to [out=160,in=20] (3.7,3);
\draw [<->,red] (0,-0.7) to [out=-110,in=110] (0,3.7);
\node at (0,-3.2) {$\quad$};
\end{tikzpicture}
\begin{tikzpicture}[smooth, scale=0.6]
\draw[step=1cm,gray,very thin] (-0.5,-0.5) grid (6.5,6.5);
\draw[thick] (-0.5,0) to (6.5,0);
\draw[thick] (-0.5,2) to (6.5,2);
\draw[thick] (-0.5,4) to (6.5,4);
\draw[thick] (-0.5,6) to (6.5,6);
\draw[thick] (0,-0.5) to (0,6.5);
\draw[thick] (2,-0.5) to (2,6.5);
\draw[thick] (4,-0.5) to (4,6.5);
\draw[thick] (6,-0.5) to (6,6.5);
\node at (0,0) {$\bullet$};
\node at (2,0) {$\bullet$};
\node at (4,0) {$\bullet$};
\node at (6,0) {$\bullet$};
\node at (0,2) {$\bullet$};
\node at (2,2) {$\bullet$};
\node at (4,2) {$\bullet$};
\node at (6,2) {$\bullet$};
\node at (0,4) {$\bullet$};
\node at (2,4) {$\bullet$};
\node at (4,4) {$\bullet$};
\node at (6,4) {$\bullet$};
\node at (0,6) {$\bullet$};
\node at (2,6) {$\bullet$};
\node at (4,6) {$\bullet$};
\node at (6,6) {$\bullet$};
\draw [<->,red] (-0.7,6) to [out=160,in=20] (6.7,6);
\draw [<->,red] (0,-0.7) to [out=-110,in=110] (0,6.7);
\end{tikzpicture}
\end{center}
\caption{The dual intersection complex of two degenerations of $E\times E$ differing by base change. Each vertex corresponds to a copy of $\P^1\times\P^1$ in the central fiber.
Scaling $\varphi$ by a factor of 2 on $B$ corresponds to dilating the decomposition by a factor of 2. Locally, the family is given by the cone over the dual intersection complex.}
\end{figure}
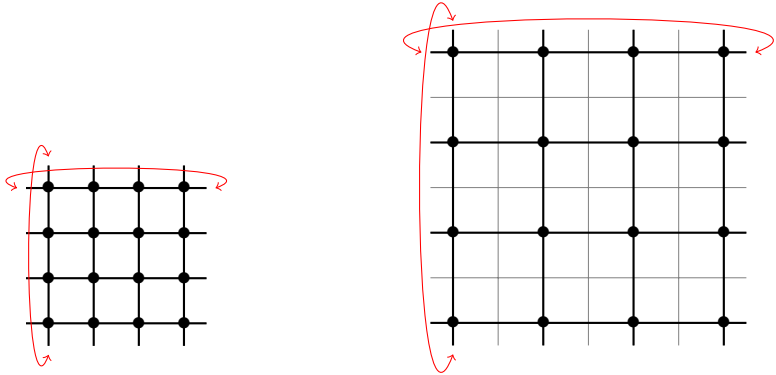

Denote by $\Sigma'$, resp.\ $\Sigma$, the fan for the family of Example \ref{eexe}, resp.\ Example \ref{eexestrict} and by $\check{\PP}'$, resp.\ $\check{\PP}$, the respective polyhedral decompositions of $N_\R$. Then $\check{\PP}$ is obtained from $\check{\PP}'$ by refining the decomposition, i.e.\ $\Sigma$ is a refinement of $\Sigma'$. In other words, the family $X_\Sigma$ is a log modification of the family $X_{\Sigma'}$.

\begin{figure}[h]
\begin{center}
\begin{tikzpicture}[smooth, scale=1.2]
\draw[step=1cm,gray,very thin] (-0.5,-0.5) grid (3.5,3.5);
\draw[thick] (-0.5,0) to (3.5,0);
\draw[thick] (-0.5,1) to (3.5,1);
\draw[thick] (-0.5,2) to (3.5,2);
\draw[thick] (-0.5,3) to (3.5,3);
\draw[thick] (0,-0.5) to (0,3.5);
\draw[thick] (1,-0.5) to (1,3.5);
\draw[thick] (2,-0.5) to (2,3.5);
\draw[thick] (3,-0.5) to (3,3.5);
\draw[thick] (-0.3,0.3) to (0.3,-0.3);
\draw[thick] (-0.3,1.3) to (1.3,-0.3);
\draw[thick] (-0.3,2.3) to (2.3,-0.3);
\draw[thick] (-0.3,3.3) to (3.3,-0.3);
\draw[thick] (0.7,3.3) to (3.3,0.7);
\draw[thick] (1.7,3.3) to (3.3,1.7);
\draw[thick] (2.7,3.3) to (3.3,2.7);
\node at (0,0) {$\bullet$};
\node at (0.3,0.2) {$\scriptstyle{(0,0)}$};
\node at (1,0) {$\bullet$};
\node at (2,0) {$\bullet$};
\node at (3,0) {$\bullet$};
\node at (3.3,0.2) {$\scriptstyle{(3,0)}$};
\node at (0,1) {$\bullet$};
\node at (1,1) {$\bullet$};
\node at (2,1) {$\bullet$};
\node at (3,1) {$\bullet$};
\node at (0,2) {$\bullet$};
\node at (1,2) {$\bullet$};
\node at (2,2) {$\bullet$};
\node at (3,2) {$\bullet$};
\node at (0,3) {$\bullet$};
\node at (0.3,3.2) {$\scriptstyle{(0,3)}$};
\node at (1,3) {$\bullet$};
\node at (2,3) {$\bullet$};
\node at (3.3,3.2) {$\scriptstyle{(3,3)}$};
\node at (3,3) {$\bullet$};
\draw [<->,red] (-0.7,3) to [out=160,in=20] (3.7,3);
\draw [<->,red] (0,-0.7) to [out=-110,in=110] (0,3.7);
\end{tikzpicture}
\end{center}
\caption{The dual intersection complex of the semistable degeneration of $E\times E$ of Example \ref{eexestrict}. The location of the vertices are read off from the slopes of $\varphi$. The resulting family is a log modification of the family of Example \ref{eexe}.}
\label{fdicexestrict}
\end{figure}
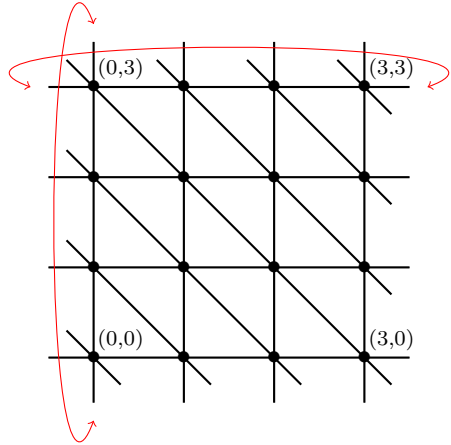

The action of $\Gamma$ on $\Xi_\varphi$ induces an action of $\Gamma$ on $\Sigma$ by
\[
\gamma : (n,p) \mapsto \left(n+(d\alpha_\gamma)^t(p),p\right).
\]
The quotient
\[
\left( X_\Sigma \times_{{\rm Spec} \, \C[t]} {\rm Spec} \, \C[t]/t^{k+1} \right) \big{/} \, \Gamma \to {\rm Spec} \, \C[t]/t^{k+1}
\]
gives the degeneration of abelian varieties $\XX \to {\rm Spec} \, \C[t]/t^{k+1}$  over a thickened point. Using \cite[Theorem 4.4]{RS20}, it can be extended to an analytic family.

Next, we describe the prelog 1-cycles of Example \ref{eexestrict} from the dual intersection complex of Figure \ref{fdicexestrict}. Denote by $\Sigma$ the corresponding fan. The cone over the dual intersection complex is actually the tropicalization of the family. Moreover, we can look at the tropicalization of a family of curves in the family. 
By functoriality, the tropicalization of the family of curves maps to the dual intersection complex.
In other words, a family of curves limits to both prelog 1-cycles in the prelog Chow group of the central fibre, and tropical curves in the dual intersection complex. A priori, tropical curves are decorated by homology classes at their vertices, but since the components of the central fiber are toric, its Chow groups are identified with singular homology.

Consider a family of curves in the degeneration that over a very general fibre specialises to a generator of ${\rm CH}_1(E\times E)$. This family determines both an element of the prelog Chow group of the central fibre, and a tropical curves. In fact, let $\Sigma_E$ be the fan in $\R^2$ whose rays are generated by $(m,1)$ for $m\in\Z$, which by projection to the second component admits a natural morphism to the fan of $\A^1$. Taking the quotient by the action that identifies the rays $\overline{(m,1)}$ and $\overline{(m',1)}$ if $m-m'\in3\Z$, we obtain the degeneration of a genus 1 curve into a cycle of three $\P^1$. At the level of the fans, we have three natural maps inducing vertical, horizontal and diagonal morphisms $X_{\Sigma'}\to X_\Sigma$, which descend to the quotient. See Figure \ref{fdegofcurves} for the corresponding tropical curve.

\begin{figure}[h]
\begin{center}
\begin{tikzpicture}[smooth, scale=1.2]
\draw[step=1cm,gray,very thin] (-0.5,-0.5) grid (3.5,3.5);
\draw[thick] (-0.5,0) to (3.5,0);
\draw[thick] (-0.5,1) to (3.5,1);
\draw[thick] (-0.5,2) to (3.5,2);
\draw[thick] (-0.5,3) to (3.5,3);
\draw[thick] (0,-0.5) to (0,3.5);
\draw[thick] (1,-0.5) to (1,3.5);
\draw[thick] (2,-0.5) to (2,3.5);
\draw[thick] (3,-0.5) to (3,3.5);
\draw[thick] (-0.3,0.3) to (0.3,-0.3);
\draw[thick] (-0.3,1.3) to (1.3,-0.3);
\draw[thick] (-0.3,2.3) to (2.3,-0.3);
\draw[thick] (-0.3,3.3) to (3.3,-0.3);
\draw[thick] (0.7,3.3) to (3.3,0.7);
\draw[thick] (1.7,3.3) to (3.3,1.7);
\draw[thick] (2.7,3.3) to (3.3,2.7);
\node at (0,0) {$\bullet$};
\node at (0.3,0.2) {$\scriptstyle{(0,0)}$};
\node at (1,0) {$\bullet$};
\node at (2,0) {$\bullet$};
\node at (3,0) {$\bullet$};
\node at (3.3,0.2) {$\scriptstyle{(3,0)}$};
\node at (0,1) {$\bullet$};
\node at (1,1) {$\bullet$};
\node at (2,1) {$\bullet$};
\node at (3,1) {$\bullet$};
\node at (0,2) {$\bullet$};
\node at (1,2) {$\bullet$};
\node at (2,2) {$\bullet$};
\node at (3,2) {$\bullet$};
\node at (0,3) {$\bullet$};
\node at (0.3,3.2) {$\scriptstyle{(0,3)}$};
\node at (1,3) {$\bullet$};
\node at (2,3) {$\bullet$};
\node at (3.3,3.2) {$\scriptstyle{(3,3)}$};
\node at (3,3) {$\bullet$};
\draw [<->,red] (-0.7,3) to [out=160,in=20] (3.7,3);
\draw [<->,red] (0,-0.7) to [out=-110,in=110] (0,3.7);
\draw[thick] (-4,-0.5) to (-4,3.5);
\node at (-4,0) {$\bullet$};
\node at (-4,1) {$\bullet$};
\node at (-4,2) {$\bullet$};
\node at (-4,3) {$\bullet$};
\draw [<->,red] (-4,-0.7) to [out=-110,in=110] (-4,3.7);
\draw[->, line width=0.5mm] (-3.3,1.5) -- (-1,1.5);
\end{tikzpicture}
\end{center}
\caption{A morphism of dual intersection complexes induced by a morphism of fans, giving all generators of $\rm{CH}^{\rm prelog}_1(\XX_0)$.}
\label{fdegofcurves}
\end{figure}
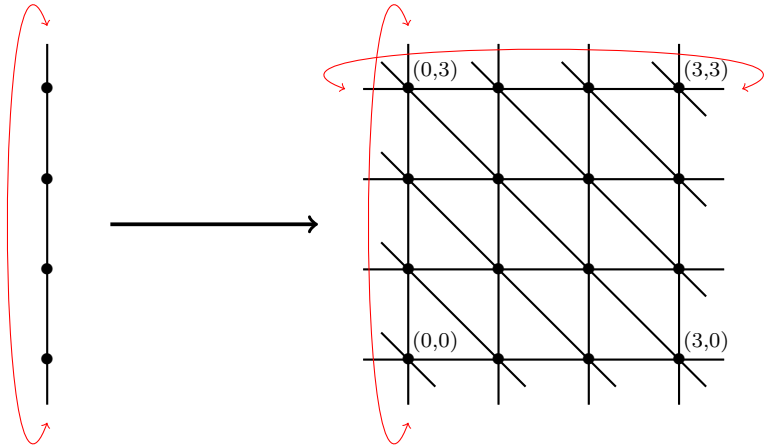

The behavior of the prelog Chow ring of this family under base change is informative.
Let $C=\mathrm{Spec}\, \C [[t]]$, $C' = \mathrm{Spec}\, \C [[s]]$, and consider the base change $C'\to C$ given by $t=s^r$ for some positive integer $r$. Denote by $\VV\to C$ the degeneration of the elliptic curve $E$ into a cycle of 3 lines. Denote by $\XX$ the semistable modification given by the quotient of the fan of Figure \ref{fdicexestrict}, with central fiber $X$:
\[
\xymatrix{
\XX \ar[r] \ar[rd] & \VV \times_C \VV \ar[d] \\
 & C 
}
\]
Then we can consider the base change
\[
\xymatrix{
\XX' = \XX\times_C C'\ar[d] \ar[r] & \XX \ar[d]\\
C' \ar[r] & C 
}
\]
which corresponds to multiplying the kinks of Figure \ref{fstrictlypolydecomp2torus} by a factor of $r$ and to dilating the dual intersection complex of Figure \ref{fdicexestrict} by a factor of $r$. We see that the total family is not smooth. This is remedied by subdividing further in the natural way, see Figure \ref{fdicexestrictbasechange} for the case $r=2$, leading to a new central fiber $X''_r$. Then $X''_r$ consists of $9 r^2$ copies of the degree 6 del Pezzo surface glued together in a way analogous to Figure \ref{fstrictlypolydecomp2torus}. Clearly no new prelog cycles have been introduced.

\begin{figure}[h]
\begin{center}
\begin{tikzpicture}[smooth, scale=1.9]
\draw[very thick] (-0.25,0) to (3.25,0);
\draw[very thick] (-0.25,1) to (3.25,1);
\draw[very thick] (-0.25,2) to (3.25,2);
\draw[very thick] (-0.25,3) to (3.25,3);
\draw[very thick] (0,-0.25) to (0,3.25);
\draw[very thick] (1,-0.25) to (1,3.25);
\draw[very thick] (2,-0.25) to (2,3.25);
\draw[very thick] (3,-0.25) to (3,3.25);
\draw[very thick] (-0.2,0.2) to (0.2,-0.2);
\draw[very thick] (-0.2,1.2) to (1.2,-0.2);
\draw[very thick] (-0.2,2.2) to (2.2,-0.2);
\draw[very thick] (-0.2,3.2) to (3.2,-0.2);
\draw[very thick] (0.8,3.2) to (3.2,0.8);
\draw[very thick] (1.8,3.2) to (3.2,1.8);
\draw[very thick] (2.8,3.2) to (3.2,2.8);
\draw[thick] (-0.25,0.5) to (3.25,0.5);
\draw[thick] (-0.25,1.5) to (3.25,1.5);
\draw[thick] (-0.25,2.5) to (3.25,2.5);
\draw[thick] (0.5,-0.25) to (0.5,3.25);
\draw[thick] (1.5,-0.25) to (1.5,3.25);
\draw[thick] (2.5,-0.25) to (2.5,3.25);
\draw[thick] (-0.2,0.7) to (0.7,-0.2);
\draw[thick] (-0.2,1.7) to (1.7,-0.2);
\draw[thick] (-0.2,2.7) to (2.7,-0.2);
\draw[thick] (0.3,3.2) to (3.2,0.3);
\draw[thick] (1.3,3.2) to (3.2,1.3);
\draw[thick] (2.3,3.2) to (3.2,2.3);
\node at (0,0) {$\bullet$};
\node at (0.2,0.1) {$\scriptstyle{(0,0)}$};
\node at (1,0) {$\bullet$};
\node at (2,0) {$\bullet$};
\node at (3,0) {$\bullet$};
\node at (3.2,0.1) {$\scriptstyle{(6,0)}$};
\node at (0,1) {$\bullet$};
\node at (1,1) {$\bullet$};
\node at (2,1) {$\bullet$};
\node at (3,1) {$\bullet$};
\node at (0,2) {$\bullet$};
\node at (1,2) {$\bullet$};
\node at (2,2) {$\bullet$};
\node at (3,2) {$\bullet$};
\node at (0,3) {$\bullet$};
\node at (0.2,3.1) {$\scriptstyle{(0,6)}$};
\node at (1,3) {$\bullet$};
\node at (2,3) {$\bullet$};
\node at (3.2,3.1) {$\scriptstyle{(6,6)}$};
\node at (3,3) {$\bullet$};
\draw [<->,red] (-0.35,3) to [out=160,in=20] (3.35,3);
\draw [<->,red] (0,-0.35) to [out=-110,in=110] (0,3.35);
\end{tikzpicture}
\end{center}
\caption{Base change corresponds to a dilation and resolution to a subdivision.}
\label{fdicexestrictbasechange}
\end{figure}
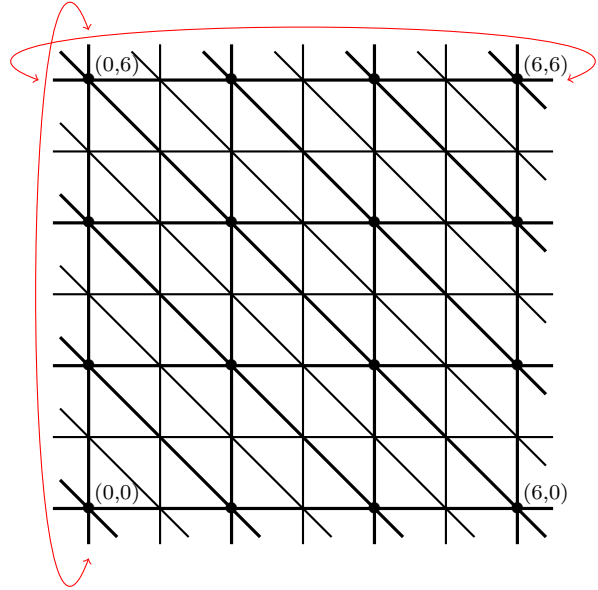

In higher dimensions, using the same method, we can describe specializations of cycles that are lower-dimensional abelian varieties, as torically described prelog cycles of the special fiber. Here we use the basic fact that toric subvarieties of toric varieties can be read off from the fan.

For a Mumford degeneration of an abelian surface, the \emph{tropical lifting problem} is studied by Nishinou in \cite{Ni20} with a full solution in the case of 3-valent tropical curves. This is the question whether a tropical curve in the dual intersection complex can be lifted to a family of degenerating holomorphic curves with tropicalization the given tropical curve. Studying curves by their tropicalization is of course a much finer classification problem than studying them by their class in the prelog Chow group.

In all dimensions, having constructed the family, \cite{GHS20} obtain sections of (powers of) the ample line bundle that the construction comes with. These are the \emph{theta functions}, which up to a factor agree with the classical theta functions. The theta functions are indexed by $B(\frac{1}{d}\Z)$ and their multiplication is given by the multiplication of broken lines. A careful analysis reveals the equations of the family in a high dimensional projective space. We refer the reader to \cite{GHS20} for details.

\providecommand{\bysame}{\leavevmode\hbox to3em{\hrulefill}\thinspace}
\providecommand{\href}[2]{#2}

\end{document}